\documentclass[a4paper,11pt]{article}
\usepackage{amscd,amssymb,amsthm}
\usepackage{palatino}
\usepackage{color}
\newcommand{\jacobi}[2]{\left( \begin{array}{c} #1\\ \hline  #2 \end{array} \right)}
\newcommand{\legendre}[2]{(  #1 | #2  )}

\newcommand{\sty}{\displaystyle}

\newtheorem{theorem}{Theorem}

\newtheorem{lemma}{Lemma}

\newtheorem{proposition}{Proposition}

\newenvironment{singlespace}
{ \small \normalsize}

\begin{document}
%
%\title{Elliptic Curves in the Representation of Primes \\     by Binary Quadratic Forms}   %

\title{On the Representation of Primes by Binary Quadratic Forms, \\
  and Elliptic Curves}   %

\author{Michele Elia\thanks{Polytechnic of Turin, Italy}~~ and~~
        Federico Pintore\thanks{University of Trento, Italy} }

%\date{}

\maketitle

\thispagestyle{empty}

\begin{abstract}
\noindent
It is shown that, under some mild technical conditions, representations of prime numbers by binary quadratic forms can be computed in polynomial complexity  
by exploiting Schoof's algorithm, which counts the number of $\mathbb F_q$-points of an elliptic curve over a finite field $\mathbb F_q$.
Further, a method is described which computes representations of primes from reduced quadratic forms by means of the integral roots of polynomials over $\mathbb Z$.
Lastly, some progress is made on the still-unsettled general problem of deciding which primes are represented by which classes of quadratic forms of given discriminant.
\end{abstract}


%\vspace{1mm}

\section{Introduction}
Let $Q(x,y)=ax^2+bxy+cy^2$ be a binary quadratic form having integer coefficients, and with discriminant $\Delta=b^2-4ac$.
 Although the problem of finding integral solutions of the  equation 
\begin{equation}
   \label{main1}
    ax^2+bxy+cy^2=m,   \hspace{5mm}  m \in \mathbb Z \,\,
\end{equation} 
was already considered by Diophantus \cite{diophantus,bashmakova}, it is not yet entirely settled  \cite{gauss,buch,buell,cox,mathews,venkov}. 
 The long story of elliptic curves also began with Diophantus, \cite{bashmakova}
and is still continuing with the discovery of ever-new and intriguing properties. The merging  of the theories of these two 
algebraic objects, namely quadratic forms and elliptic curves, has generated a large number of results, whose  importance extends beyond number theory, and that also have practical applications \cite{buch,koblitz,hoff}.\\
The problem of deciding whether equation (\ref{main1}) is solvable depends on the discriminant $\Delta$ of the  quadratic form, and on $m$, and was in part addressed by Lagrange  \cite{lagrange}, who  proved the following
\begin{lemma}{\cite[Lemma 2.5, p. 26]{cox}}
   \label{lagrange}
An odd integer $m$ is represented by some quadratic form of discriminant $\Delta$, with $\gcd\{\Delta,m\}=1$, if and only if $\Delta$ is a quadratic residue modulo $m$, that is, $\Delta$ is a quadratic residue modulo every prime factor of $m$.
\end{lemma}

\noindent
The problem of solving equation (\ref{main1}), which is also known as the \textit{representation problem} for the integer $m$ and the quadratic form $Q(x,y)$ of discriminant $\Delta$ may be split into two parts:
\begin{quotation}
\noindent
\begin{tabular}{ll}
 {\bf Problem 1:}  &
   a) {\em Establish whether  $Q(x,y)$ can represent $m$}, \\   
  &  b) {\em Find a representation whenever $Q(x,y)$ represents $m$}.
\end{tabular}
\end{quotation}

\noindent
The theory of quadratic forms, established by Gauss in his {\em Disquisitiones Arithmeticae} \cite{gauss}, sets the problem in its proper perspective.
 Gauss's theory, by introducing the notions of classes of quadratic forms and composition of forms, reduces the  representation problem of any composite $m$ to the representation of its prime factors \cite[p. 74-75]{buell}. This connection with factoring makes quadratic form theory of foremost importance, in view of its cryptographic applications \cite{hoff, koblitz, menezes}. 

\noindent
A first step towards resolving the representation problem, assuming the decomposition of $m$ into prime factors to be known, is to obtain the representations of primes by quadratic forms of discriminant $\Delta$. 
These representations are important for many reasons, but fundamentally because every prime might be represented only by a single class of quadratic forms
 \cite{buell,cox,mathews, cohn}. 
Consequently, when more than a single class exists, a second important question is:
\begin{quotation}
\noindent
\begin{tabular}{ll}
  {\bf Problem 2:} & {\em Find which class represents which prime}.
\end{tabular}
\end{quotation}

\noindent
 An elementary solution to this problem appears not to be possible, since neither integer congruences nor genus theory (i.e. Jacobi characters) are sufficient for the purpose. In particular, as discussed in Cox's book \cite{cox}, to establish whether a prime $p$ is represented by a principal form it is necessary to proceed by splitting polynomials modulo $p$. 
Further, there are still many computational obstacles before explicit representations may be obtained.\\
In this direction, Schoof's computational strike \cite{schoof}, by exploiting the connection between quadratic forms and elliptic curves, allows us to solve Problem 1 when $m$ is prime, in polynomial complexity (possibly under some technical condition), the complexity depending on the size of $m$. 

\noindent
The paper is organized as follows. Section 2 summarizes the background of quadratic forms, 
  to put this subject into context. 
Section 3 deals with the connections between ideals, quadratic fields, and elliptic curves. 
Section 4 addresses some computational aspects for computing square roots modulo prime numbers, in particular using Schoof's algorithm to count the number of points of elliptic curves over finite fields. 
Section 5 presents a new algorithm to compute the representation of primes by reduced quadratic forms.
Section 6 comments on the partition of primes into classes of representability, with some examples.
 Section 7 is devoted to conclusions and comments on future work. 

\section{Preliminaries}
Throughout this paper, primitive binary quadratic forms $Q(x,y)$ alone will be considered. A form $ax^2+bxy+cy^2$ will also be denoted $(a,b,c)$, and its discriminant $\Delta=b^2-4ac$ will always be assumed to be either square-free, or $4$ times a square-free integer; set $D=\Delta$ if  $\Delta=1 \bmod 4$, and $D=\frac{\Delta}{4}$ if $\Delta=0 \bmod 4$. 
Let $\mathcal B=\{1, \omega \}$ denote an integral basis of the quadratic field  
$\mathbb K=\mathbb Q(\sqrt D)$,
 then $\omega$ can be taken to be $\omega=\frac{1+\sqrt D}{2}$ if $D =1 \bmod 4$, and 
 $\omega=\sqrt{D}$ if $D=2,3 \bmod 4$. Let $\mathfrak G(\mathbb K/\mathbb Q)=\{e, \sigma \}$ denote the Galois group of  $\mathbb K$ over $\mathbb Q$.  
 Two forms $Q_1(x,y)$ and $Q_2(x,y)$ with the same discriminant $\Delta$ are equivalent if integers $p$, $q$, $r$, and $s$ exist such that~ 
$ Q_1(x,y)=Q_2(px+qy,rx+sy) ~\mbox{and}~  ps-qr=\pm 1,  ~ $
and are properly equivalent if $ps-qr=1$, \cite[\S 157]{gauss}. 
That is, all  properly equivalent quadratic forms are produced by the action of the special unimodular group  
 $PSL(2, \mathbb Z)$ of $2$-dimensional matrices on one of them \cite[Theorem 3.7, p.116]{mollin1}.
The action of this group partitions
the set of quadratic forms with the same discriminant $\Delta$ into a finite set of $h_{\mathbb{K}}$ proper equivalence classes \cite[Theorem 3.7, p.116]{mollin1}. 
The class number  $h_{\mathbb K}$ is equal to the number of ideal classes of $\mathbb K$ \cite{frohlich}. 
Each class of properly equivalent quadratic forms is identified by some reduced form
 $ax^2+bxy+cy^2$, that is, forms whose coefficients satisfy the condition
\begin{equation}
  \label{redcon}
      \begin{array}{cl}
         |b| \leq a\leq  c   &  ~~~\mbox{if}~~\Delta < 0 \\
       0 < b < \sqrt {\Delta}  ~~~ \mbox{and} ~~~ 
       \sqrt \Delta -b < 2 |a| < \sqrt \Delta +b & ~~~\mbox{if}~~\Delta > 0 ~~~. 
       \end{array}   
\end{equation}
Positive definite reduced forms have $|b| \leq \sqrt{\frac{|\Delta|}{3}}$
 \cite[Proposition 2.1, p.13]{buell}; furthermore, distinct reduced forms are not properly
  equivalent, \cite[Theorem 2.4, p.15]{buell}, with two possible exceptions: 
$$        ax^2+bxy+ay^2 \sim ax^2-bxy+ay^2   ~~~~\mbox{and}~~~~  
       ax^2+axy+cy^2 \sim ax^2-axy+cy^2 ~~.  \\
$$      
In either of these two cases, the equivalence class representative is chosen so as to have a non-negative center coefficient, consequently if $\Delta < 0$ every form is equivalent to a single reduced form (see \cite[p.17]{buell}). \\
Whereas if $\Delta > 0$, each class $\mathcal C_i$, $i=1, \ldots, h_{\mathbb K}$, of properly equivalent reduced forms
consists of an even number $P_i(\Delta)$ of forms \cite{gauss,venkov}, usually  $P_i(\Delta)>2$, when these cardinalities $P_i(\Delta)$ are large, 
the representation problem is more laborious. The explicit value of the period of a class $\mathcal C_j$ 
is given as a theorem, which is taken without proof from \cite[p.111]{venkov}, see also \cite{perron}. 
\begin{theorem}
  \label{period}
The period of a class $\mathcal C_j$ is equal to the period of the continued fraction representing the positive root $\Omega$ of the quadratic characteristic polynomial $at^2+bt+c$ associated to
 any quadratic form $(a,b,c)$ in class $\mathcal C_j$. 
\end{theorem}

\noindent
It follows that the length of any period is upper bounded by $\Delta^{\/2} \ln(\Delta)$,  \cite[p.329-337]{hua}). 
 Since every known method for computing these periods is of non-polynomial complexity in the size of the discriminant, the direct continued fraction development of $\Omega$ is  a non-polynomial. \\
It should be remarked that, although of comparable size, the periods may be quite different. As an example, consider the three classes (or periods) 
$\mathcal C_1$, $\mathcal C_2$, and $\mathcal C_3$ of reduced quadratic forms of discriminant
$4 \cdot 7565$, which correspond to the  three classes of ideals in the quadratic field $\mathbb Q(\sqrt{7565})$ of class number $h_{\mathbb K}=3$. 
These classes are fully reported in Table \ref{tab1}, along with  the periods of the continued fraction of $\Omega$ in the columns labelled $\alpha$,
 containing, the positive roots of the characteristic polynomial associated with the first quadratic form in each class (see \cite[p.127]{dirichlet}). Furthermore, the coefficients in the columns with header the values of $\Omega$ define the linear transformations 
$$  \begin{array}{lcl} 
     (a_i,b_i,c_i)  \rightarrow   (a_i, b_i+2\alpha_i a_i, c_i+b_i \alpha_i+ a_i \alpha_i^2)  &~~&  \mbox{odd} ~~ i \\        
  (a_i,b_i,c_i)  \rightarrow   ( a_i+b_i \alpha_i+ c_i \alpha_i^2, b_i+2\alpha_i c_i, c_i)  &~~&  \mbox{even} ~~i \\
    \end{array}
$$
from the quadratic form on the same line to the quadratic form on the line below. The transformation of the quadratic form on the last line of each period returns to the initial quadratic form.

\noindent
The notions of proper equivalence and proper representation are nicely connected
 by the following lemma, quoted from \cite[p.25]{cox} without proof.

\begin{lemma}
   \label{lemmacox}
A form $Q(x,y)$ properly represents an integer $m$ if and only if $Q(x,y)$ 
 is properly equivalent to the form $m x^2+Bxy+Cy^2$ for some $B, C \in \mathbb Z$.   
\end{lemma} 

\noindent
An immediate consequence of this Lemma is Lagrange's result, given above as Lemma \ref{lagrange}.
It is worth recalling that different classes of improperly equivalent quadratic forms represent the same set of primes. 
However, Cheboratev's theorem implies that the set of primes represented by the quadratic forms of the same discriminant is equally partitioned, 
 in the sense that the relative density of primes represented by one
  out of the $h_{\mathbb K}$ properly inequivalent classes is $\sty \frac{1}{h_{\mathbb K}}$.

\begin{table}
  \begin{tabular}{|c|c|c|c|c|c|c|}  \hline
 & $\frac{86+\sqrt{7565}}{1}$  &  $\mathcal C_1$ & $\frac{85+\sqrt{7565}}{5}$  &  $\mathcal C_2$ &
                                                                  $\frac{83+\sqrt{7565}}{13}$  &  $\mathcal C_3$ \\ \hline
$i$ &$\alpha$  &  $(a,b,c)$ & $\alpha$  &  $(a,b,c)$ &
                                                                  $\alpha$  &  $(a,b,c)$ \\ \hline
 $1$ & $172$        & $[1,-172,-169]$  &  $34$     & $[5,-170,-68]$   &   $13$         & $[13,-166,-52]$   \\
 $2$ & $1$            & $[-169,172,1]$   &  $2$       & $[-68,170,5]$     &   $13$         & $[-13,172,13]$   \\
 $3$ &  $42$         & $[4,-166,-169]$  &  $1$       & $[73,-102,-68]$ &   $3$           & $[52,-166,-13]$   \\
 $4$ &  $2$           & $[-85,170,4]$     &  $1$       & $[-97,44,73]$     &   $3$          & $[-43,146,52]$   \\ 
 $5$ &  $42$         & $[4,-170,-85]$    &  $8$       & $[20,-150,-97]$  &   $1$         & $[103,-112,-43]$   \\
 $6$ &  $1$           & $[-169,166,4]$   &  $10$     & $[-17,170,20]$    &   $2$         & $[-52,94,103]$   \\  \hline
 $7$ &                   &                             &  $8$       & $[20,-170,-17]$  &   $1$         & $[83,-114,-52]$   \\
 $8$ &                   &                             &  $1$       & $[-97,150,20]$   &   $1$         & $[-83,52,83]$   \\
 $9$ &                   &                             &  $1$       & $[73,-44,-97]$    &   $2$         & $[52,-114,-83]$   \\
 $10$ &                 &                             &  $1$       & $[-68,102,73]$   &   $1$         & $[-103,94,52]$   \\  \hline
 $11$ &                 &                             &               &                          &   $3$         & $[43,-112,-103]$   \\
 $12$ &                 &                             &               &                          &   $3$         & $[-52,146,43]$   \\  \hline
  \end{tabular}
\caption{Reduced classes of quadratic forms with  $\Delta =4\cdot 7565$}
\label{tab1}
\end{table}

%  [172, 1, 42, 2, 42, 1];  [34, 2, 1, 1, 8, 10, 8, 1, 1, 2]; [13, 13, 3, 3, 1, 2, 1, 1, 2, 1, 3, 3]
%

%
\noindent
 Gauss's composition of two forms $Q_1(x,y)=(a_1, b_1, c_1)$  and 
   $Q_2(x,y)=(a_2,b_2,c_2)$, having the same discriminant $\Delta$, produces a quadratic form
  $Q_3(x,y)=(a_3,b_3,c_3)$ with the same discriminant and such that
$ Q_1(x_1,y_1) Q_2(x_2,y_2) = Q_3(x_3,y_3) $, if $x_3$ and $y_3$ are chosen as suitable
 bilinear functions of the pairs of variables $(x_1,y_1)$ and $(x_2,y_2)$.
Several composition methods are known: Appendix A reports
 Arndt's algorithm from \cite{buell}  without proof.  \\
Gauss's composition gave the set of classes of reduced forms a group structure, which turns out to be isomorphic to the class group of $\mathbb Q(\sqrt{\Delta})$, that is, isomorphic to the class group of
 field ideals.

\subsection{Genera}
Let $r$ be the number of different odd prime divisors of $D$. For each $D$, and every $m$ relatively prime to $D=\prod_{i=1}^r q_i$, the Jacobi characters are defined as
\begin{equation}
  \label{jacochar}
  \chi_i(m) =\jacobi{m}{q_i}   \hspace{5mm} \left\{  \begin{array}{l} 
      i=1, 2,\ldots , r   \hspace{5mm} , \hspace{5mm} D =1 \bmod 4   \\
      i=~~~~ 2, \ldots , r   \hspace{5mm} , \hspace{5mm}  D = 2,3 \bmod 4 
     \end{array}   \right. \,\,.
\end{equation}
When $D =2,3 \bmod 4$, the missing character $\chi_1(m)$  is defined as follows
 (see \cite[Lemma 5, p.253]{cohn})
$$   \chi_1(m) = \left\{   \begin{array}{lcl}
          \legendre{-1}{|m|} \mbox{sgn}(m) & ~~ &  4D = 12  \bmod 16  \\
          \legendre{2}{m}  & ~~ &  4D = 8  \bmod 32  \\
          \legendre{-2}{|m|} \mbox{sgn}(m)  & ~~ &  4D = 24  \bmod 32  \\
     \end{array}  \right.      \,\,, $$
here $ \mbox{sgn}(m)$ means "the sign of $m$", and when $D <0$, it is understood that $m >0$. \\
The finite set of  classes of quadratic forms is further partitioned into a finite set of subsets, called genera.
Each genus is identified by the same set of Jacobi characters, that is a block of $r$ consecutive $\pm$ signs. Every prime $p$ yielding the same set of characters (\ref{jacochar}), which may be rewritten as $\{ \chi_1(p), \ldots, \chi_r(p)\}$, is represented by some class of quadratic forms in the same genus.
The following theorem of Gauss's specifies the main property of the partition of the class group into genera, (see \cite[Theorem 4, p.234]{cohn} for a proof).

\begin{theorem}              %[Gauss - 1801]
   \label{gaussth}
If we consider the $h_{\mathbb K}$ proper equivalence classes of forms with discriminant equal to a field discriminant $\Delta$, then they can be subdivided equally into $2^{r-1}$ genera of $\frac{h_{\mathbb K}}{2^{r-1}}$ classes in each genus.
\end{theorem}

\noindent
Genus theory is undoubtedly useful to tackle Problem 2, since it reduces the search to a search within genera. In this connection, the worst case occurs when  $\Delta$ is prime, that is  $r=1$, since there is a single genus with $h_{\mathbb K}$ classes.
It will be seen that the Hilbert class field may be useful to make some distinction within each genus.
In particular, if $\frac{h_{\mathbb K}}{2^{r-1}} \leq 3$, further subdivisions can be obtained by means of the Hilbert class polynomial.
The unresolved cases occur when $\sty \frac{h_{\mathbb K}}{2^{r-1}}$ has some prime factor greater than $3$.

\section{Quadratic fields, Ideals, Elliptic curves}
  \label{sectEC}
The correspondence between quadratic forms and ideals of quadratic fields offers
 an alternative approach to the composition of forms,
 which explicitly discloses the group structure of the classes of quadratic forms. 
Let $a_1, a_2, \ldots, a_s$ be elements of a field  $\mathbb F$. An ideal 
$\mathfrak I=\langle a_1, a_2, \ldots, a_s \rangle$
 of $\mathbb F$ is defined as the set 
 $$\langle a_1, a_2, \ldots, a_s \rangle = \{ a_1x_1+ a_2x_2 +\ldots + a_sx_s:~x_1, \ldots, x_s \in \mathbb O_{\mathbb F}\} \,\,.$$
When we consider a quadratic field $\mathbb K$, any ideal $\mathfrak I$  is identified by a pair of elements
 of $O_{\mathbb K}$, and written as $\langle a, b \rangle$. If $\mathfrak I = \langle  a, ab \rangle$, it is specified by the single element $a$, and is called principal ideal.
Since any principal ideal is of the form $\langle a, ab \rangle$  
and consists of the multiples of a single element
  $a \in \mathbb K$, it is also represented as
  $\langle a \rangle = \{ a x:~x \in \mathbb O_{\mathbb K}\}$. \\
The product of two ideals is defined as
$ \langle a, b \rangle \langle c, d \rangle = \langle ac, ad, bc, bd \rangle $
 and the following simplification rules can be used to reduce the expression with four terms to
 the standard form $\langle A, B \rangle$:
 $$   \begin{array}{l}
     \langle a_1, a_2, a_3, 0 \rangle = \langle a_1, a_2, a_3 \rangle \\
     \langle a_1, a_2, a_3, a_4 \rangle = 
        \langle a_1, a_2+\lambda a_1, a_3+\lambda a_1, a_4+\lambda a_1 \rangle ~~~~ \lambda \in \mathbb K\\
      \end{array}  ~~.
 $$
Since, by definition, $\langle a, b \rangle =\langle b, a \rangle $, and 
$\langle a, b \rangle = \langle a, b+\lambda a \rangle$, with $\lambda \in \mathbb K$,
 these rules can be iteratively applied to produce a canonical form for the ideal
 $\langle a, b \rangle =\langle e \rangle\langle f, g+\sqrt D \rangle$, with $f,g$
 positive rational integers satisfying the condition $g < f$ \cite{edwards}.  
Given an ideal $\mathfrak I=\langle a, b \rangle$ the elements of $\mathbb K$ can be
 partitioned into a finite number of classes. Each class, denoted $c+\mathfrak I$,
 is defined as the set
$$c+\mathfrak I = \{c+ a x+ b  y:~x,y \in \mathbb O_{\mathbb K}\}~~.$$
The norm $N_{\mathbb K}(\mathfrak I)$ of an ideal $\mathfrak I$ is defined as the
 number of distinct classes. When $\mathfrak I$ is represented in canonical form
 $\langle e \rangle\langle f, g+\sqrt D \rangle$, the norm is computed as $|e|f$.  
Using the notion of product, two ideals $\mathfrak I$ and $\mathfrak L$ are
  equivalent if two principal ideals exist such that  $ \langle a \rangle \mathfrak I = \langle b \rangle \mathfrak L$. Thus, all principal ideals
  form a single class $\mathfrak C_0$, and the non-principal ideals are partitioned into
 equivalence classes $\mathfrak C_i$, each class having a representative ideal 
 $\mathfrak I_i =\langle f, g+\sqrt D \rangle$ which is usually chosen to be of minimum norm, i.e. $|f|$
takes the minimum value.  

\begin{proposition}
   \label{ideal1}
Every ideal $\mathfrak I=\langle  \mathfrak a, \mathfrak b  \rangle$
 is associated to an integral primitive quadratic form $Q_{\mathfrak I}(x,y)$ as
\begin{equation}
   \label{qfdef} 
   Q_{\mathfrak I}(x,y)= \frac{\mathfrak a \sigma(\mathfrak a) x^2 +
      (\mathfrak  a\sigma(\mathfrak b)+ \mathfrak b\sigma(\mathfrak a) )x y +
      \mathfrak b \sigma(\mathfrak b)y^2}{N_{\mathbb K}(\mathfrak I)}~~. 
%%%%      ~~~~~~i=2,\ldots ,h_{\mathbb K}
\end{equation}
\end{proposition}

%\noindent
\begin{proof}
The coefficients of $Q_{\mathfrak I}(x,y)$ are integers, because
 $\mathfrak I$ strictly divides both $\langle  \mathfrak a \rangle$ and
 $\langle  \mathfrak b \rangle$  by definition, thus $\mathfrak I$ divides
 $\langle  \mathfrak a+\mathfrak b \rangle$, and it follows that 
 $N_{\mathbb K}(\mathfrak I)$ strictly divides $\mathfrak a \sigma(\mathfrak a)$,
 $\mathfrak b \sigma(\mathfrak b)$,
         and $(\mathfrak  a \sigma(\mathfrak b)+ \mathfrak b \sigma(\mathfrak a) )$.
Note that, in this correspondence between quadratic forms and ideals, 
 the product of ideals corresponds to the composition of the corresponding quadratic forms 
 \cite{cohen2}. 
\end{proof}

\noindent
A property of this correspondence between ideals and quadratic forms is expressed as a lemma. 
 
\begin{lemma}[\cite{cohen2}]
   \label{lem1}
Equivalent ideals are associated, through equation (\ref{qfdef}), to  
 quadratic forms of the same class.
\end{lemma}

%\noindent

\begin{proof} 
An ideal, equivalent to $\mathfrak I=\langle \mathfrak a, \mathfrak b \rangle$, is obtained
 as $\langle \mathfrak e \rangle \mathfrak I$ for $\mathfrak e \in \mathbb K$. 
The conclusion follows because 
 $N_{\mathbb K}(\langle \mathfrak e \rangle\mathfrak I) =
 \mathfrak e \sigma(\mathfrak e)N_{\mathbb K}(\mathfrak I)$
  is a direct consequence of $\mathfrak  a \mathfrak  f \sigma(\mathfrak b\mathfrak f)=
     \mathfrak a \sigma(\mathfrak b)\mathfrak f\sigma(\mathfrak f)$ and 
  the chain of identities
$$ N_{\mathbb K}(\langle  \mathfrak f \rangle\mathfrak I)=
   N_{\mathbb K}(\langle  \mathfrak f \rangle)N_{\mathbb K}(\mathfrak I)=
   \mathfrak f\sigma(\mathfrak f) N_{\mathbb K}(\mathfrak I) ~~.$$
\end{proof}

\subsection{Ideals and Hilbert class fields}
%.  
The Hilbert class field $\mathbb L$ of $\mathbb K$ is an unramified
 extension of degree $h_{\mathbb K}$ such that every non-principal
 ideal of $\mathbb K$ becomes principal \cite[Theorem 4.18, p.189]{nark}. In 
quadratic fields $\mathbb K=\mathbb Q(\sqrt{\Delta})$, it is well known that rational primes either ramify, remain inert, or split into the product of two ideals, which may be principal or non-principal. The following theorem specifies exactly what occurs in the Hilbert class field, and may be seen as a corollary of \cite[Theorem 5.5, p.391]{neukirch},
 or as a re-formulation of \cite[Corollary 4.121, p.250]{mollin} in a form useful to the aims of this paper. 
 
\begin{theorem}
   \label{artin}
 Every rational prime $p$ that splits in $\mathbb K$ (i.e. the ideal $\langle p \rangle$
  splits into a pair of principal ideals) fully splits
  in the Hilbert class field of $\mathbb K$.
\end{theorem}

\noindent
The field $\mathbb L$ is specified by a root
 $\alpha$ of an irreducible polynomial $h_{\mathbb K}(x)$ of degree $h_{\mathbb K}$ over $\mathbb K$, the Hilbert class polynomial.
The Galois group $\mathfrak G(\mathbb L/\mathbb K)$ of $\mathbb L$ over
 $\mathbb K$ is isomorphic to the ideal class group, thus it is Abelian and coincides
 with the Galois group of the Hilbert class polynomial $h_{\mathbb K}(x)$ with respect to $\mathbb K$; in particular, $h_{\mathbb K}(x)$ is solvable by radicals. 
 $\mathbb L$ is a normal extension of $\mathbb Q$ defined by a root
 of an irreducible polynomial $H_{\mathbb K}(x)$ over $\mathbb Q$ of degree $2h_{\mathbb K}$. 
In particular, dealing with imaginary quadratic fields, there is an interesting connection between lattices, elliptic curves, elliptic functions, and a special Hilbert polynomial defining $\mathbb L$.  

\subsection{Imaginary quadratic forms and Elliptic curves}
Let $\{ 1, \omega \}$ be an integral basis for $\mathbb K$. Consider the lattice $\Gamma (\Delta) = \{ n+ m \omega | n, m \in \mathbb Z \}$, which is identified with the maximal order $\mathfrak O_{\mathbb K}$, and is left invariant by the modular $PSL(2,\mathbb Z)$ of $2 \times 2$ matrices with integer coefficients and unit determinant.
%$$      M= \left(  \begin{array}{cc}
%                   a & b  \\
%                   c  & d  
%           \end{array}  \right)  \hspace{5mm} \forall  \,\,  a,b,c,d \in \mathbb Z, \,\,\,  \mbox{with} \,\,\, 
%     ad-bc =1  \,\,.$$
Besides this natural group of endomorphisms, the lattice is also left invariant by a proper complex factor $\theta$, that is $\theta \Gamma(\Delta)=\Gamma(\Delta)$; i.e.
 $\{ 1, \omega \}$ and  $\{ \theta, \theta \omega \}$ generate the same lattice. \\
The lattice is linked to the Weierstrass function $\wp(z)$, which is a doubly periodic function, that is
$\wp(z+1)=\wp(z+\omega)=\wp(z)$, and satisfies the differential equation
$  \left( \frac{d\wp(z)}{dz} \right)^2=  4\wp(z)^3-g_2\wp(z)-g_3 $.
This equation shows that $\wp(z)$ can be used to parametrize the elliptic curve 
 $\mathbb E(\mathbb C|\jmath_0)$ of equation
$$  y^2 =4x^3-g_2x-g_3  \hspace{5mm}    x,y \in \mathbb C  \,\,,     $$
defined over the complex field $\mathbb C$, i.e.  $(x,y)=(\wp(z), \wp(z)')$.  The constants $g_2$ and $g_3$ depend
only on the lattice $\Gamma(\Delta)$, thus they are invariant under the endomorphisms of the lattice, and may be used to define  two special invariants: namely, the form $ D(\Gamma)= g_2^3-27g_3^2$, which is the discriminant of the cubic polynomial divided by $16$,
and the  $\jmath$-invariant $ \jmath(\Gamma) = \frac{1728 g_2^3}{D(\Gamma)}$, which is also
invariant under certain scale transformations of the elliptic curve. The factor $1728$
in the definition of $\jmath$ serves to make the coefficients in the following equation into integers,
which is useful for its computation  \cite[p.86]{chandra}, \cite{weber}, or \cite{schoof}:
%
%$$ \jmath =  \frac{1}{8q^2}  \left( \prod_{n=1}^{\infty} (1+q^{2n-1})^{16} + \prod_{n=1}^{\infty} (1-
% q^{2n-1})^{16} +2^8 q^2  \prod_{n=1}^{\infty} (1+q^{2n})^{16}  \right)  $$
%which can be simplified as
%
 \begin{equation}
   \label{Jchandra}
  \jmath  = \frac{1}{q^2} + 744  +  \sum_{n=1}^{\infty}  b_n q^{2n}  \hspace{5mm}   b_n \in \mathbb Z  \,\,,  
 \end{equation} 
and $z=\omega$ in $q=e^{2\pi i z} $ is related to the complex multiplier. \\
The elliptic curve is connected, via the lattice $\Gamma(\Delta)$, to the  quadratic form $x^2+(\omega+\bar \omega) xy+\omega\bar \omega y^2 $
which is principal, and defines a metric on the lattice.
The links between imaginary quadratic  fields, quadratic forms, and  elliptic curves  are illustrated in Figure \ref{fig1}. 
 Two points should be noted:
\begin{enumerate}
 \item The equation of  an  elliptic curve
 $\mathbb E(\mathbb L|\jmath_0)$  with complex multiplication in $\mathfrak O_{\mathbb K}$, 
and with assigned $\jmath$-invariant $\jmath_0 \in \mathbb L$, may be written as
 \begin{equation}
   \label{jinv}
   y^2= 4x^3 + \frac{27\jmath_0}{1728-\jmath_0} x +
                \frac{27\jmath_0}{1728-\jmath_0}   ~~~  ~~~,
%   y^2= x^3 + \frac{3\jmath_0}{1728-\jmath_0} x +
%                \frac{2\jmath_0}{1728-\jmath_0}  ~~~,
 \end{equation}  
if  $\jmath_0 \neq 0,  ~1728$; while if the $\jmath$-invariants are $0$ and $1728$, the elliptic curves,  clearly defined over $\mathbb Q$,  have equations $y^2=4x^3+1$ and $y^2=4x^3+x$, respectively. 
  \item The $\jmath$-invariants are algebraic numbers that are roots of the Hilbert class polynomial $h_{\mathbb K}(x)$, a peculiar polynomial defining the Hilbert class field $\mathbb L$ of $\mathbb K$. 
The $\jmath$-invariants can be computed from equation (\ref{Jchandra}) and may be approximated 
with sufficient precision, as described in  \cite{schoof},  with complexity $O(1)$ with respect to $\Delta$.  Schoof also showed that 
$h_{\mathbb K}$ can be computed in $O(\Delta^{1+\epsilon})$, and the Hilbert polynomial 
 $h_{\mathbb K}(x)$ can be computed with complexity $O(\Delta^{2.5+\epsilon})$ for every $\epsilon > 0$. 
\end{enumerate}

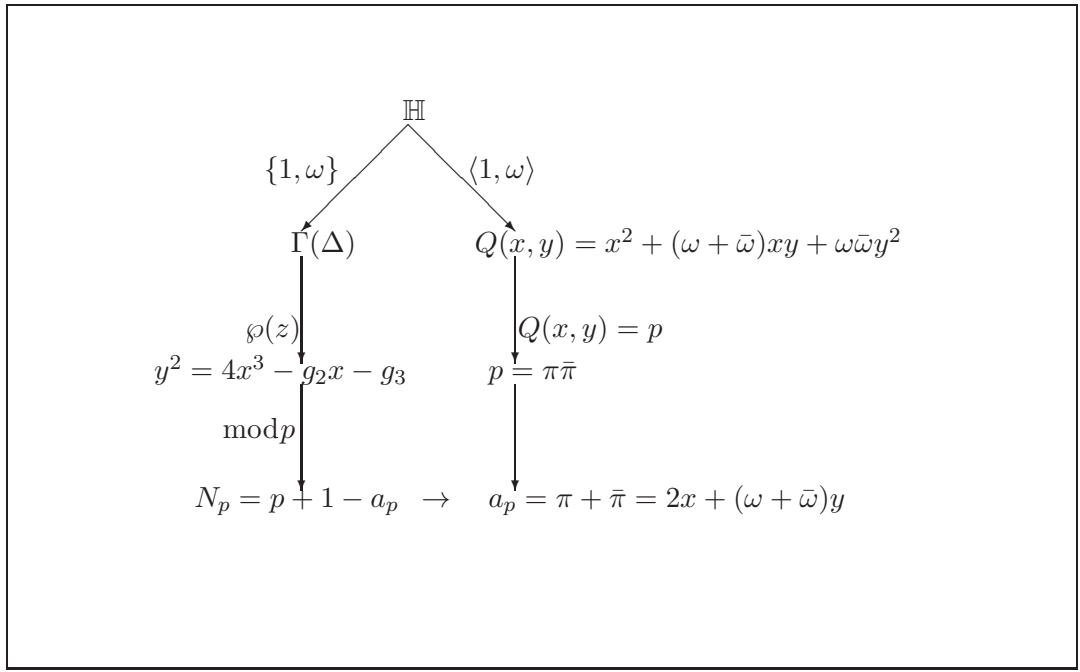
\begin{figure}
\begin{picture}(100,200)(0,0)
\put(45,5){\framebox(400,250)}
\put(193,212){$\mathbb H$}

\put(195,210){\vector(1,-1){40}}
\put(195,210){\vector(-1,-1){40}}
\put(217,190){\parbox{30mm} {$\langle 1,\omega \rangle$} }
\put(141,190){\parbox{30mm}{$\{1,\omega \}$} }
\put(151,162){$\Gamma(\Delta)$}
\put(220,162){$Q(x,y)=x^2+(\omega+\bar \omega)xy+\omega \bar \omega y^2$}
\put(155,160){\vector(0,-1){40}}
\put(235,160){\vector(0,-1){40}}
\put(134,130){$\wp(z)$}
\put(236,130){$Q(x,y)=p$}
\put(100,114){$y^2=4x^3-g_2x-g_3$}
\put(225,114){$p=\pi \bar \pi$}
\put(155,112){\vector(0,-1){40}}
\put(235,112){\vector(0,-1){40}}
\put(125,92){$\bmod p$}
%\put(236,92){$Q(x,y)=p$}
%
\put(115,66){$N_p=p+1-a_p$}
\put(225,66){$a_p=\pi + \bar \pi=2x+(\omega+\bar \omega)y$}
\put(200,66){$\rightarrow$}

\end{picture}
\caption{Connections between elliptic curves and quadratic forms via imaginary quadratic fields}
\label{fig1}
\end{figure}

\section{Solving $Q(x,y)=m$}
Given an odd prime $p$ such that $\legendre{\Delta}{p}=1$, a form
 with discriminant $\Delta$ trivially representing $p$ is obtained as  $\left( p, b_1, \frac{b_1^2-\Delta}{4p} \right)$, where $b_1$ is
 a solution of the modular equation $b^2= \Delta  \bmod p$ such that $b_1^2-\Delta$ is divisible by $4$. 

\paragraph{Square roots.}
Consider the computation of a square root of  $ \Delta  \bmod p$. 
When $p = 3 \bmod 4$, the task is feasible with deterministic polynomial complexity, since we have 
$\sqrt \Delta = \Delta^{\frac{p+1}{4}} \bmod p$. When $p = 1 \bmod 4$ the task is harder. 
However, an efficient probabilistic method is offered by the Cantor-Zassenhaus
   algorithm for factoring polynomials over finite fields \cite{cantor}, which performs the task with complexity $O(\log_2 p)$. 
Alternatively, assuming that the Hilbert class polynomial of $\mathbb Q(\sqrt \Delta)$ is known, the square root
can be computed with deterministic polynomial complexity of order $O((\log_2 p)^9)$ \cite{schoof} by means of Schoof's algorithm, which counts the number of rational points of an elliptic curve over a finite field
$\mathbb F_q$. For the sake of easy reference, the way in which Schoof's algorithm is exploited is briefly recalled.\\   
Assume that $\Delta$ is negative and let $\mathbb{L}$ be the Hilbert class field of the imaginary quadratic field $\mathbb{Q}(\sqrt{\Delta})$. An  elliptic curve $\mathbb E(\mathbb L|\jmath_0)$ modulo a prime
 ideal  $ \mathfrak p \subset \mathcal{O}_{\mathbb{L}}$, which is a factor of the ideal $p \mathcal{O}_{\mathbb{L}}$, is an
 elliptic curve  $\mathbb E(\mathcal{O}_{\mathbb L}/\mathfrak p|\jmath_0)$ over the Galois
 field $\mathbb F_q=\mathcal{O}_{\mathbb L}/\mathfrak p$ of remainders modulo $\mathfrak p$. Let $\kappa$ denote a prime ideal factor of $p  \mathcal{O}_{\mathbb K}$, and let $\ell_{\mathfrak p}$  be the minimum factor of $h_{\mathbb K}$ such that 
 $\kappa^{\ell_{\mathfrak p}}$ is a principal ideal in $\mathbb K$. 
The order of $\mathbb F_q$ is equal to  $N_{\mathbb L}(\mathfrak p)=p^{\ell_{\mathfrak p}}$.
When $\ell_{\mathfrak p} >1$, the Hilbert class polynomial, or one of its irreducible factors modulo $p$ (of degree $\ell_{\mathfrak p}$), is used to define $\mathbb F_q$.
The number of $\mathbb F_q$-points on $\mathbb E(\mathbb F_q|\jmath_0)$ 
 is $N_{q}=|E(\mathbb F_q|\jmath_0)|=q+1-a_{q}$, with $|a_q| \leq 2\sqrt{q}$. 
A theorem of Deuring's \cite[Theorem 14.16, p. 317]{cox} or
  \cite[Ch.13, Theorem 12]{lang} establishes the existence of an element $\pi \in \mathcal{O}_{\mathbb K}$ such that 
   $q=\pi \bar{\pi}$, and the identity
$ N_q = q+1-(\pi+\bar{\pi})= (\pi-1)(\bar{\pi}-1) $,
 thus we have
\begin{equation}
   \label{deuring}
 a_q=\pi+\bar{\pi}  ~.
\end{equation} 
  Writing $\pi = a+ b \omega$, it is immediately seen that $2a+b(\omega+\bar \omega)=a_q$, which,  together with $q=a^2+(\omega+\bar \omega)ab+(\omega \bar \omega)b^2$,
  allows $a$ and $b$ to be computed as
%\end{document}

\begin{equation}
  \label{class1}
(a,b)= \left\{  \begin{array}{lcl}
  \sty  \left(\frac{a_q-b}{2} \hspace{1mm},\hspace{1mm}\sqrt{\frac{4q-a_q^2}{-\Delta}}~ \right) &~~ &  \Delta = 1 \hspace{1mm} \bmod \hspace{1mm} 4  \\
   \\ 
\sty    \left(\frac{a_q}{2} \hspace{1mm},\hspace{1mm} \sqrt{\frac{4q-a_q^2}{-\Delta}} ~ \right)    &~~ &  \Delta = 0 \hspace{1mm} \bmod \hspace{1mm} 4  \\
   \end{array}  \right.
\end{equation}
Notice that, considering the expression for $b$ modulo $p$, 
we get $\sqrt{\Delta}  \bmod p$ as
$$     \sqrt{\Delta}   = \pm \frac{a_{q}}{b}    \bmod p   \,\,.$$
{\bf Remark.} When $\Delta$ is positive, this formula can be used to compute $\sqrt{-\Delta} \bmod p$. Further, since 
$p=1 \bmod 4$, the equation $p=x_o^2+y_o^2$ gives $\sqrt{-1}$ as $\frac{x_o}{y_o} \bmod p$, and finally  $\sqrt{\Delta} = \frac{x_o}{y_o} \sqrt{-\Delta} \bmod p$ is obtained. 

\vspace{3mm}
\noindent
{\bf Remark.} In his book \cite{cox}, David Cox debates the general problem of establishing whether $p$
 is represented by a principal quadratic form of the kind $x^2+n y^2$, where $n$ is a natural number possibly divisible by squares. 
His conclusion \cite[Theorem 9.2, p. 180]{cox} is that $p$ is represented by a principal form of
  discriminant $\Delta < 0$ if and only if $\legendre{\Delta}{p}=1$, 
  and there is an irreducible monic polynomial $f_{\Delta}(x)$ with integer coefficients
 that has a linear factor modulo $p$.
The above arguments illustrate the situation also when $p$ is representable by a non-principal quadratic form. The conclusions are more general than Cox's, but are limited to (square-free) field discriminants.

%%%%%%%%%%%

\subsection{An algorithm}
Solving equation (\ref{main1}), i.e. solving Problem 1, is a different matter compared to the problem solved in the previous section, because the quadratic form is given.  However, an algorithm can be devised \cite{venkov} which is based on the above procedure and  Gauss's reduction algorithm \cite{mathews}. \\
Given a quadratic form $ax^2+bx y+cy^2$ and an integer  $m=\prod_{i=1}^s p_i^{\alpha_i}$ decomposed into its prime power factors, the following algorithm returns a representation $(x_o,y_o)$ of $m$ by $(a,b,c)$, or a failure if $m$ cannot be represented.
\begin{center}
{\bf    Algorithm   G}
\end{center}
\begin{description}
  \item[Step 1:] Via Gauss's reduction algorithm, find a reduced quadratic form $AX^2+BX Y+CY^2$ properly equivalent to $ax^2+bx y+cy^2$ and the corresponding linear transformation $T_1$ $$     
  \left[  \begin{array}{c}   
                       x   \\
                       y
           \end{array}    \right]  =
           \left[  \begin{array}{cc}   
                       \alpha  &  \beta  \\
                      \gamma & \delta
           \end{array}    \right]
\left[  \begin{array}{c}   
                       X   \\
                       Y
           \end{array}    \right].
$$
  \item[Step 2:] For every $i$ between $1$ and $s$, 
     \begin{itemize}
       \item[1.] Find two reduced forms, $(a_i,b_i,c_i)$ and $(a_i',b_i',c_i')$, consisting of a pair of reciprocal forms, representing $p_i$ (actually the form is only one if the two forms belong to the same proper equivalence class), and the corresponding representations, $(x_i,y_i)$ and $(x_i',y_i')$, with positive $x$-coordinates.
       \item[2.] Compute all compositions of $\alpha_i$ forms consisting of $k$  forms   $(a_i,b_i,c_i)$ and $\alpha_i -k$ forms $(a_i',b_i',c_i')$, for every $k$ from $0$ to $\alpha_i$, and find the corresponding representations of $p_i^{\alpha_i}$. The resulting forms are exactly $\alpha_i+1$.
    \end{itemize}
  \item[Step 3:]   Compose $s$ reduced forms, one per every $p_i^{\alpha_i}$, in all possible ways, and enumerate each resulting composed form by the index $\ell$, to get
  the forms $g_\ell x^2+f_\ell x y+h_\ell y^2$, and obtain all possible representations
$(x^{(\ell)}_m,y^{(\ell)}_m)$ of $m$  by these forms  ( see Appendix A).
    \item[Step 4:] For every $\ell$ find a reduced quadratic form $G_{\ell}X^2+F_{\ell}X Y+H_{\ell}Y^2$ properly equivalent to $g_\ell x^2+f_\ell x y+h_\ell y^2$, and the corresponding representation  $(X^{(\ell)}_m,Y^{(\ell)}_m)$  of $m$ by this reduced form. 
    \item[Step 5:] Find an $\ell$, if one exists, such that $G_{\ell}X^2+F_{\ell}X Y+H_{\ell}Y^2$ is properly or improperly equivalent to $Ax^2+Bxy+Cy^2$ and the corresponding linear transformation $T_2$,       that is
$$     \left[  \begin{array}{c}   
                      X  \\
                       Y
           \end{array}    \right]  =
           \left[  \begin{array}{cc}   
                       \alpha_2 &  \beta_2  \\
                      \gamma_2 & \delta_2
           \end{array}    \right]
\left[  \begin{array}{c}   
                       x  \\
                       y
           \end{array}    \right]   ~~.
$$ 
       Otherwise STOP and output FAILURE.  
\item[Step 6:]  Combine $T_1$ and $T_2$, and finally obtain the representation of $m$ by 
   $ax^2+bx y+cy^2$
as
$$     \left[  \begin{array}{c}   
                      x_o  \\
                       y_o
           \end{array}    \right]  =
         T_1  T_2^{-1} 
\left[  \begin{array}{c}   
                       X^{(\ell)}_m  \\
                       Y^{(\ell)}_m
           \end{array}    \right] ~~.
$$
\end{description}

\noindent
Gauss's reduction algorithm may found reported in a clever form in \cite[p. 69-74 ]{mathews}, 
distinguishing between positive and negative $\Delta$, since the algorithm forms are slightly different. \\
The Algorithm G works for both positive and negative discriminants, although, for the case of positive discriminants, it may require a burden of computations to determine whether a reduced form $G_{\ell}X^2+F_{\ell}X Y+H_{\ell}Y^2$ stays in the same equivalence class of $Ax^2+Bx y+Cy^2$, which is trivial for negative discriminants. \\
Algorithm G gives at least a representation of $m$, provided that one exists. In general, the
 number of representations of $m$  is larger than $2$ when $m$ is composite.

\section{An alternative method} 
   \label{sect5}
The connection between binary quadratic forms and elliptic curves, briefly mentioned above, provides the setting necessary to describe an alternative, and in some ways more direct, method for solving the equation $Q(x,y)=p$, with $p$ an odd prime and $Q(x,y)$ a reduced quadratic form of discriminant $\Delta$, when solvable (i.e. solve part 2 of Problem 1). The method exploits Proposition \ref{ideal1}.  \\
Assume that $h_{\mathbb K}$ reduced quadratic forms $Q_0(x,y),Q_1(x,y),\dots,Q_{h_{\mathbb{K}}-1}(x,y)$, with $\mathbb{K}=\mathbb{Q}(\sqrt{\Delta})$ representative of each proper equivalence class of quadratic forms of  discriminant $\Delta$, where $Q_0(x,y)$ denotes the principal form. The following theorem allows us to find representations of $p$ in a different way from the procedure based on Gauss's reduction.

\begin{theorem}
\label{lab2}
If the odd prime $p$ is represented by some reduced form $Q_i(x,y)$ of discriminant $\Delta$ are known, the ideal $p\mathcal{O}_{\mathbb{K}}$, with 
$\mathbb{K}=\mathbb{Q}(\sqrt{\Delta})$, splits into two ideals, principal or non-principal, $\langle p \rangle = \mathfrak P \bar \mathfrak P$ and two integers $x_0,y_0$, such that $Q_i(x_0,y_0)=p$, can be obtained from a Diophantine system of two equations of degree $\ell_{p}$
$$  \left\{  \begin{array}{l}
           a_1(x,y) =u   \\
           a_2(x,y) =v   \\
    \end{array}   \right.  \,\,
$$
 where $\ell_{p}$ is the minimum divisor of $h_{\mathbb{K}}$ such that $ \mathfrak P^{\ell_{p}}$ is principal, and $a_1(x,y)$, $a_2(x,y)$ are homogeneous polynomials of 
 degree $\ell_{p}$. 
The two integers $u,v$ are a solution of $p^{\ell_{p}}=Q_0(u,v)$, the sign of $u$ being selected appropriately. 

\end{theorem}
\begin{proof}
Let $\langle 1,\omega \rangle$ be the principal ideal of $\mathcal{O}_{\mathbb{K}}$ associated to the principal form $Q_0(x,y)$ which also defines the norm in $\mathbb Q(\sqrt{\Delta})$, and let $\mathfrak{I}_{a_1},\dots,\mathfrak{I}_{a_{h_{\mathbb{K}}-1}}$, with $\mathfrak{I}_{a_i}=\langle a_i, \beta_i +\omega \rangle$, be the ideals of $\mathcal{O}_{\mathbb{K}}$ having norm $a_i$ \cite[Chapter 6]{buell}, associated to $Q_1(x,y),\dots,Q_{h_{\mathbb{K}}-1}(x,y)$ respectively.
Since there is an isomorphism between the ideal class group of $\mathcal{O}_{\mathbb{K}}$ and the proper equivalence classes of forms of discriminant $\Delta$, it follows 
that $q=p^{\ell_{p}}$ is properly represented by the principal form $Q_0(x,y)$, i.e. $q=\pi \overline{\pi}$ for some $\pi=u+\omega v \in \mathcal{O}_{\mathbb{K}}$. 
This representation can be obtained directly using equation (\ref{class1}) if $\Delta <0$. Otherwise it can be obtained by reducing the quadratic form $qx^2+2bxy+\frac{b^2-\Delta}{q}y^2$, that represents $q$ trivially, to the norm form $x^2+(\omega+\bar \omega)xy+\omega \cdot \bar \omega y^2$. The coefficient $b$ is the square root of $\Delta$ modulo $q$ and can be obtained by the Hensel lifting of the square root of
$\Delta$ modulo $p$, which can be easily obtained as described by the remark  immediately following equation  (\ref{class1}). The reduction is straightforward since the criterion is to get a quadratic form with minimum middle coefficient, i.e. either $0$ or $1$. \\
Recalling equation (\ref{qfdef}), the non-principal quadratic form $Q_i(x,y)$ can be written as a ratio of two norms
$$   Q_i(x,y)= \frac{ N_{\mathbb K} (a_ix+(\beta_i+\omega)y)}{N_{\mathbb K} (\mathfrak I_{a_i})} ~~;  $$
from this it follows that
$$  Q_0(u,v)= N_{\mathbb K} (u+\omega v)= (Q_i(x_0,y_0))^{\ell_{p}}= \frac{ N_{\mathbb K} ([a_ix_0+(\beta_i+\omega)y_0]^{\ell_{p}})}{N_{\mathbb K} (\mathfrak I_{a_i}^{\ell_{p}})} ~~ $$
for suitable $x_0,y_0 \in \mathbb{Z}$.
Observing that  
 $\pi_{a_i} \bar \pi_{a_i}=a_i^{\ell_{p}}$, we obtain the identity 
$$   u+\omega v =  \frac{ [a_ix_0+(\beta_i+\omega)y_0]^{\ell_{p}}}{\pi_{a_i}} = \frac{\bar \pi_{a_i} [a_ix_0+(\beta_i+\omega)y_0]^{\ell_{p}}}{a_i^{\ell_{p}}} ~~,   $$
where $\pi_{a_i} \in \mathcal{O}_{\mathbb K}$ is a generator of the principal ideal $\mathfrak I_{a_i}^{\ell_{p}}$  whose norm is $a_i^{\ell_{p}}$. 
An explicit expression for $\pi_{a_i}$ is obtained by computing the $\ell_{p}$-power
of $\mathfrak I_{a_i} =\langle a_i, \beta_i+\omega \rangle$. 
 In conclusion, a representation $Q_i(x_0,y_0)$ of $p$ is obtained by solving a Diophantine system
 of two equations in two unknowns, i.e. by computing the integer root of a polynomial.
\end{proof}

%%%%%%
\section{Problem 2 with small $h_{\mathbb K}$}
    \label{exemp}
The solution of Problem 2 presents marked differences between negative and positive discriminants,
 differences that will be analyzed separately before proving a theorem covering all cases. 

\paragraph{Positive discriminants.}
  \label{realT}
When the discriminant $\Delta$ is positive, the class number $h_{\mathbb K}$ of the quadratic field $\mathbb{K}=\mathbb{Q}(\sqrt{\Delta})$ is greater than $1$, and the number of reduced quadratic forms in each proper equivalence class is greater than $2$, the problem of determining whether a quadratic form is principal can be tackled in several ways. Two possible approaches  will be described: the first is based on continued fractions, the second on Theorem \ref{pprinform} below.

\paragraph{1.}
The periodic continued fraction representing $\sqrt{\Delta}$ is commonly written as 
$         [d_0, [d_1, d_2, \ldots , d_T]]   $,
where $d_0=\lfloor \sqrt \Delta \rfloor$ is the anti-period and the entries between the inner brackets constitute the period of length $T$  ($1\leq T \leq \Delta^{1/2} \log \Delta$
 \cite[p.329-337]{hua}), with $d_T=2d_0$.
Let $\frac{p_i}{q_i}$ be the partial quotients, also called convergents, of the continued fraction. 
Numerators and denominators of the convergents are computed recursively as
$$  \left\{  \begin{array}{lcl}
          p_i = d_ip_{i-1}+p_{i-2} & ~~ & p_0=d_0,  ~~ p_{-1}=1  \\
          q_i = d_iq_{i-1}+q_{i-2} & & q_0=1,  ~~ q_{-1}=0  \\
      \end{array}   \right.   ~~ i=1,2, \ldots    ~~.
$$
The sequence $\mathcal S=\{\Delta_i=p_i^2-\Delta q_i^2\}_{i=1}^{\infty}$ satisfies the following properties, see  \cite{hua}:  
\begin{enumerate}
 \item  $\mathcal S$ is periodic with  period $T$.
 \item $|\Delta_i| < 2 \sqrt{\Delta}$ for every $i$.
 \item $\Delta_{T-1} = (-1)^T$, i.e.   $p_{T-1}+q_{T-1} \sqrt{\Delta}$ is the fundamental unit in $\mathbb Q(\sqrt{\Delta})$.
 \item All integers of absolute value less than $\sqrt{\Delta}$ which are represented by the principal forms occur in a period of the sequence $\mathcal S$.
\end{enumerate}

\noindent
The above properties offer a criterion for testing whether a quadratic form is principal \cite{mathews}.
\begin{proposition}
A quadratic form $(a,b,c)$, with positive discriminant $\Delta$, is principal if and only if one of the coefficients $a_r$ or
 $c_r$ of any reduced form $(a_r,b_r,c_r)$ among the $T$ reduced forms occurs in a period of the sequence $\mathcal S$ constructed from  the continued fraction of $\sqrt \Delta$.
\end{proposition}

\paragraph{2.}
The second criterion is a consequence of the following theorem.

\begin{theorem}
  \label{pprinform}
Let $(a,b,c)$  be a quadratic form with discriminant $\Delta>0$, and $\mathbb K=\mathbb Q(\sqrt \Delta)$ be a real quadratic field whose Hilbert class field  $\mathbb L$ is defined by the root of a known polynomial
 $h_{\mathbb{K}}(x)$  of degree $h_{\mathbb K}$ over $\mathbb K$.
Suppose that all prime factors $q_i$ occurring in $\mbox{\emph{lcm}}\{a ,c\}$ are known, then $(a,b,c)$ is principal if
 $h_{\mathbb{K}}(x)$ fully splits modulo $q_i$ for every $i$.
\end{theorem}

\begin{proof}
If $h_{\mathbb{K}}(x)$ fully splits modulo $q_i$, then $q_i$ is representable by a principal form, because $q_i$ splits into
 $2h_{\mathbb K}$ prime factors in $\mathbb L$ (see \cite[p.137-138]{dedekind}), and thus into two conjugate prime factors in $\mathbb K$. Hence $q_i$ is representable by a principal quadratic form. The composition of forms implies that $a$ and $c$ are representable by a principal form, which in turn implies that $(a,b,c)$ is principal.
\end{proof}

\paragraph{Negative discriminants.}
Recalling that the number of imaginary quadratic fields with given class number $\mathbf h$ is finite (a circumstantial proof of this, which Gauss hypothesized, is given in \cite{watkins,goldfeld}), 
the complete list of fields for class number $1$ is given in Table \ref{tabclass1} below, while the compleste lists \cite{sloane} for  class numbers $2$, and $3$ are given as Tables \ref{roma2}, \ref{tab3a}, \ref{tab4a}, \ref{tab5}, and \ref{tab6}.

\begin{table}[h]
\begin{center}
\begin{tabular}{|c|c|c|c|c|c|}    \hline
$D$ & $\Delta$ & $\omega$                        & $Q_0$                  &  $\jmath_0$          & $\mathbb{E}(\mathbb{L} \mid j_0)$   \\  \hline
 -1   & $-4$       & $\sqrt{-1}$                       & $x^2+ y^2$          &  $12^3$            & $x^3-x$   \\ 
 -2   & $-8$       & $\sqrt{-2}$                       & $x^2+2y^2$         &  $20^3$            & $x^3-\frac{375}{98}x-\frac{125}{49}$   \\  \hline
 -3   & $-3$       & $\frac{1+\sqrt{-3}}{2}$    & $x^2+xy+y^2$     &  $0$                  & $x^3-1$      \\
 -7   & $-7$       & $\frac{1+\sqrt{-7}}{2}$    & $x^2+xy+2y^2$   &  $(-15)^3$         & $x^3-\frac{125}{63}x-\frac{250}{189}$  \\
 -11 & $-11$     & $\frac{1+\sqrt{-11}}{2}$  & $x^2+xy+3y^2$   &  $(-32)^3$         & $x^3-\frac{1536}{539}x-\frac{1024}{539}$   \\ 
-19  & $-19$     & $\frac{1+\sqrt{-19}}{2}$  & $x^2+xy+5y^2$   &  $(-96)^3$         & $x^3-\frac{512}{171}x-\frac{1024}{513}$   \\  
-43  & $-43$     & $\frac{1+\sqrt{-43}}{2}$  & $x^2+xy+11y^2$ &  $(-960)^3$       & $x^3-\frac{512000}{170667}x-\frac{1024000}{512001}$   \\  
-67  & $-67$     & $\frac{1+\sqrt{-67}}{2}$  & $x^2+xy+17y^2$ &  $(-5280)^3$     & $x^3-\frac{85184000}{28394667}x-\frac{170368000}{85184001}$   \\  
-163& $-163$   &$\frac{1+\sqrt{-163}}{2}$ & $x^2+xy+41y^2$ &  $(-640320)^3$ & $ x^3-\frac{151931373056000}{50643791018667}x-\frac{303862746112000}{151931373056001}$ \\  \hline
  
\end{tabular}

\vspace{5mm}

Legend 

\vspace{2mm}
\begin{tabular}{lcl} \hline
$\Delta$ &~~& Field discriminant \\
$\omega$ &~~& Integral basis element \\
$Q_0$ &~~& Reduced form \\
$\jmath_0$  &~~& $\jmath$-invariant,~  root of the Hilbert class polynomial \\
$\mathbb{E}(\mathbb{L} \mid j_0)$ &~~& Elliptic curve of given $\jmath$-invariant  \\ \hline
\end{tabular}
\end{center}

\caption{Imaginary quadratic fields $\mathbb K= \mathbb Q(\sqrt{\Delta})$ of class number 1}
\label{tabclass1}
\end{table}

\vspace{3mm}
\noindent
Let $p$ be an odd prime represented by some quadratic form of discriminant $\Delta$. Given the quadratic field $\mathbb{K}=\mathbb{Q}(\sqrt{\Delta})$ and its Hilbert class polynomial $h_{\mathbb K}(x) \in \mathbb{Z}[x]$, when $h_{\mathbb K} \leq 6$, excluding $5$, the joint use of the Hilbert class polynomial $h_{\mathbb K}(x)$ and genus theory allows us to identify the equivalence class representing $p$ without computing the representation of $p$, as shown in Theorem \ref{splt} below.\\
To check the full factorization of $h_{\mathbb{K}}(x) \bmod{p}$ in polynomial complexity, operatively, we compute
$$a(x) = \gcd \{h_{\mathbb K}(x) \bmod p, x^{p-1}-1) \}       \,\,, $$ 
thus
 \begin{itemize}
   \item[-] If $a(x) = h_{\mathbb K}(x)$ then $p$ is represented by the quadratic
     forms of the principal class.
   \item[-] If $a(x) \neq h_{\mathbb K}(x)$ then $p$ is represented by the quadratic
     forms of some non-principal class.
\end{itemize}   

\begin{theorem}
   \label{splt} Given a discriminant $\Delta$, assume that the following are known: the class number $h_{\mathbb{K}}$, and the Hilbert class polynomial $h_{\mathbb{K}}(x)$ of the quadratic field $\mathbb K =\mathbb{Q}(\sqrt{\Delta})$, and at least one representative of each proper equivalence class of quadratic forms with discriminant $\Delta$. If $h_{\mathbb{K}} \in \{2,3,4,6\}$, then all primes representable by some quadratic form of discriminant $\Delta$ may be separated into sets, one set for each equivalence class, 
using  Jacobi characters and the Hilbert class polynomial, as follows:
\begin{description}
   \item[$h_{\mathbb K}=2$:]  there are two genera; the Jacobi characters suffice to separate all representable primes. 
   \item[$h_{\mathbb K}=3$:] there is a single genus; splitting the Hilbert class polynomial $h_{\mathbb K}(x)$ is necessary and sufficient to separate all primes represented by the principal forms and non-principal forms into two sets.
   \item[$h_{\mathbb K}=4$:] the form class group may have two group structures:
    \begin{description}
      \item[-] it may be the Vierergruppe, i.e. isomorphic to  $\mathbb Z_2 \times \mathbb Z_2$: there are four genera, and the Jacobi characters suffice to separate all representable primes.
      \item[-] it may be cyclic, i.e. isomorphic to $\mathbb Z_4$: there are two genera, and the Hilbert class polynomial $h_{\mathbb K}(x)$ is indispensable to separate all representable primes.
   \end{description}
 \item[$h_{\mathbb K}=6$:]  there are six proper equivalence classes and two genera;  representable primes can be separated into four sets by the joint use of the Hilbert class polynomial $h_{\mathbb K}(x)$
      and the Jacobi characters. 
\end{description}
\end{theorem}

\begin{proof}
The different cases, corresponding to different values of the class number, will be addressed in order:
 
\paragraph{$h_{\mathbb K}=2$.} 
Since the class number is even, the discriminant $\Delta$ is certainly composed; thus we necessarily have two genera, which are identified by different values of Jacobi characters. In this case, also splitting 
\cite[p.390]{pohst} the Hilbert class polynomial $h_{\mathbb K}(x)$ may be used to separate all primes represented by the principal and non-principal forms into two sets.
\paragraph{$h_{\mathbb K} = 3$.} Since the class number is prime, there is a single Jacobi character, which is trivial, and there is only a single genus. The separation of representable primes is obtained by  splitting the Hilbert class polynomial $h_{\mathbb K}(x)$; they are partitioned only into two sets, precisely
\begin{itemize}
   \item[-] the set of primes $p$ represented by the quadratic forms of the principal class, which are identified   by the full splitting of  $h_{\mathbb K}(x)$ modulo $p$;
  \item[-] the set of primes $p$ represented by the quadratic forms of the remaining two proper equivalence classes that are composed by forms improperly equivalent, and are identified by the fact that  $h_{\mathbb K}(x)$ is irreducible modulo $p$.
\end{itemize}

\paragraph{$h_{\mathbb K} =4$.} Since there are two non-isomorphic groups of order $4$, the cyclic group and the Vierergruppe, there could be two different kinds of class groups with different genera, correspondingly, representable primes are differently partitioned: 

\begin{itemize}
\item[-]  class group isomorphic to the Vierergruppe $\mathbb Z_2 \times \mathbb Z_2$:  each class is self-reciprocal, and there are four genera.
Thus representable primes are separated into four sets, one set for each proper equivalence class of quadratic forms, by the Jacobi characters.
\item[-]  class group cyclic of order $4$: due to the group structure there are only two 
self-reciprocal proper equivalence classes of quadratic forms and two genera, thus two Jacobi characters.
Representable primes are partitioned into three sets, using the Jacobi characters and the factorization of $h_{\mathbb{K}}(x)$. Two sets concern the principal genus composed  of the principal class and the
 self-reciprocal proper equivalence class. The third set corresponds to the second genus which  comprises  two proper  mutually reciprocal equivalence classes.\\
Alternatively, the separation can be achieved using only the splitting of the Hilbert class polynomial $h_{\mathbb K}(x)$:
\begin{itemize}
    \item[  i)] If  $h_{\mathbb K}(x)$  splits modulo $p$ into $4$ linear factors, then $p$ is represented 
      by the quadratic forms of the principal class;  
    \item[ ii)] If  $h_{\mathbb K}(x)$  splits modulo $p$ into $2$ quadratic factors, then $p$ is  represented  by the  quadratic forms of the second self-reciprocal proper equivalence class;
    \item[iii)] If  $h_{\mathbb K}(x)$ does not split modulo $p$, then $p$ is
     represented by some quadratic form of the second genus.
\end{itemize}
\end{itemize}

\paragraph{$h_{\mathbb K} =6$.} The quadratic form class group is cyclic of order $6$; there are six proper equivalence classes and two genera containing three classes each. A genus $\mathcal G_1$ contains the principal class and two proper equivalence classes that are mutually reciprocal; the second genus $\mathcal G_2$ contains a self-reciprocal class of quadratic forms and two classes that are mutually reciprocal. \\
 Representable primes are partitioned into two sets, one per genus, which are distinguished by their Jacobi characters. The set of primes pertaining to each genus may be further partitioned into two sets by splitting the Hilbert class polynomial $h_{\mathbb K}(x)$, precisely
\begin{enumerate}
   \item A prime $p$ represented by the quadratic forms belonging to genus $\mathcal G_1$ is
    represented by the quadratic forms of the principal class if $h_{\mathbb K}(x)$ fully splits modulo $p$, otherwise it is represented by the quadratic forms of the remaining two proper equivalence classes.
   \item A prime $p$ represented by the quadratic forms belonging to genus $\mathcal G_2$ is
    represented by the quadratic forms of the self-reciprocal proper equivalence class if $h_{\mathbb K}(x)$ splits modulo $p$ into three factors of degree $2$, otherwise it is represented by the quadratic forms of the remaining two classes.
\end{enumerate}

\end{proof}

%\textcolor{blue}{
\subsection{An example}
%\begin{example}
A significant example illustrating most of the issues concerning the partition of primes into representability classes is offered by the smallest positive discriminant $82$,\cite[Table III, page 271]{cohn}  of class number $h_{\mathbb K}=4$. Let $\alpha$ denote a root of $x^2-82$; the fundamental unit in $\mathbb K=\mathbb Q(\alpha)$ is $9+\alpha$ which has norm $-1$, and the class group is cyclic of order $4$. The four classes
of properly inequivalent quadratic forms are only partitioned into two genera, thus Jacobi characters cannot separate totally the whole set of representable primes, a target that is achieved using factorization of the Hilbert polynomial. \\
The non principal ideals are $\pi_2=(2,\sqrt{82})$,  $\pi_3=(3,1+\sqrt{82})$ and its conjugate $\bar \pi_3=(3,1-\sqrt{82})$, to which are associated the reduced quadratic forms
$2x^2-41y^2$, $3x^2- 2xy-27y^2$, and  $3x^2+2xy-27y^2$,  respectively. 
The Hilbert class polynomial is defined over $\mathbb K=\mathbb Q(\alpha)$ 
$$  h_{\mathbb K}(x)=x^4-(120\alpha +1676)x^2+71760\alpha+672644  \,\,. $$
Consider $H_{\mathbb K}(x) = h_{\mathbb K}(x) \bar  h_{\mathbb K}(x)$ defined over $\mathbb Q$,
the set of primes $p$ such that
$\jacobi{10}{p}=1$ is partitioned into three subsets of primes represented by the forms $x^2-82y^2$ (principal),   $2x^2-41y^2$,  and $3x^2\pm 2xy-27y^2$, respectively. The smallest primes represented by the three improperly inequivalent forms are:
$$ \begin{array}{l} 
    \mathcal P_1 = \{  73, 103, 113, 223, 359, 401, 449, \ldots   \}  ~~~\mbox{if} ~~H_{\mathbb K}(x) ~~\mbox{fully splits modulo $p$} \\
   \mathcal P_2 = \{ 23, 31, 127, 241, 271, 337, 353,  \ldots   \}  ~~~\mbox{if} ~~H_{\mathbb K}(x) ~~\mbox{splits into quadratic factors modulo $p$}  \\     %367, 409, 431, 433,
   \mathcal P_{3,4} = \{ 3, 11, 19, 29, 53, 67, 101, 109, 149, 157,  \ldots   \}   ~~~\mbox{if} ~~H_{\mathbb K}(x) ~~\mbox{splits into quartic factors modulo $p$}   %179, 181, 211, 227, 229, 293, 317, 331, 347, 397, 421,
  \end{array}
\,\,.   $$
Note that each genus contains two quadratic forms, that is $\{ x^2-142y^2,~2x^2-41y^2 \}$ is the first genus identified by the signature $ (++)$,  and $\{ 3x^2\pm 2xy-27y^2 \}$ is the second genus identified
by the signature $(-,-)$. The Hilbert class polynomial must be used to separate the primes between $\mathcal P_1$ and $\mathcal P_2$. \\
This example is completed,  using the alternative method described in Section \ref{sect5} by computing a representation of the prime $p=8081$ belonging to
$ \mathcal P_2$ which implies that $2~x^2 -41~y^2 =8081$ is solvable in $\mathbb Z$. 
The prime  $8081$ splits into two non-principal ideals of $\mathbb K$, namely, $\mathbf \pi=\langle 8081, 2737+\sqrt{82} \rangle$ and its conjugate. The ideal square $\mathbf \pi^2= 8081 \langle 1 , \sqrt{82} \rangle$
is principal, which means that $8081^2$ can be represented by the principal  quadratic form $x^2-82y^2$, norm of $\mathbb K$. A solution of   $x^2-82y^2=8081^2$ is found by starting with the principal quadratic form
$Q_a(x,y)=    8081^2 x^2 -2bxy + \frac{b^2-82}{8081^2} y^2 $
and reducing it to the form  $x^2-82y^2$. The coefficient $b$ is computed as a root of $u^2-82 \bmod 8081^2$, which is obtained by lifting a root of $u^2-82 \bmod 8081$, i.e. as $2737+t 8081$, where $t$
is computed by solving a linear equation
$$ 82 = (2737+t 8081)^2  \bmod 8081^2 \Rightarrow 0=7491087+2t 2737 8081  \bmod 8081^2   \Rightarrow  0=927+2 \cdot 2737 t \bmod 8081$$
that is, $b=2737+t~8081=5190739 $. Upon reducing $Q_a(x,y)$ to the form $x^2-82y^2$, we get
the representation of $8081^2$ as $(8819,390)$, and consequently the set of equations
$$ \left\{ \begin{array}{rcl} 
         2 x^2+41 y^2 &=& 8819    \\
         2~x~ y&=&  390
       \end{array}    \right.
$$
from which an equation satisfied by $x$ is obtained
$$  2~x^4-8819~x^2+1559025 = (x+65)~(x-65)(2~x^2-369) ~~. $$
In conclusion, the representation  $x_o=65,y_o=3$  of $8081$ is produced. \\
The computation of a representation of a prime in class $ \mathcal P_{3,4} $ may illustrate the method more clearly.  Consider the prime $p=239347$ that belongs to class $ \mathcal P_{3,4} $ since its signature is $(- ~ -)$, that is the  Jacobi characters are $\jacobi{2}{p} =\jacobi{41}{p}=-1$. The exponent $\ell_p$ is $4$, then $p^{\ell_p}$ is represented by the principal form $x^2-82y^2$: a representation is found by starting with the form
$Q_a(x,y)=   p^4 x^2 -2bxy + \frac{b^2-82}{p^4} y^2 $
and reducing it to the form  $x^2-82y^2$. The coefficient $b=631164344666839182838$ is a root of $u^2-82 \bmod p^4$,and is computed by lifting (Hensel lifting) the root $u=59431 $  of $u^2-82 \bmod p$.
Upon reduction of $Q_a(x,y)$ to the form $x^2-82y^2$, we get
the representation $\mathfrak r=(89593576091, 7607129250)$ of $p^4$, and consequently the set of equations
$$ \left\{ \begin{array}{rcl} 
        x^4+108~x^3~y-54~x^2~y^2+996~x~y^3-247~y^4 &=& -1430126783319    \\
        x^4+54~x^2~y^2-24~x~y^3+85~y^4&=&  158057739341
       \end{array}    \right.
$$
where the correct known terms are obtained considering that the representation $\mathfrak r$
is known apart from a multiplication by a unit.
From the last system an equation satisfied by $x$ is obtained
$$ \begin{array}{l}
     4864106742784~x^{16}-270701981991284512575616~x^{12}+
         5549832096231657223051561479738660~x^8- \\     
       53544176878327185627902860618323552699383569~x^4+ \\
      181137242756632602186135956055478753968726288295184656=0
\end{array}
$$
which has two rational roots $x=\pm 286$, to which correspond two $y$s, i.e. $y=\pm 145$.
% 
%\end{example}
%}

\section{Conclusions}
The computational problem of representing a prime $p$ by some reduced quadratic form
 $Q_G(x,y)$ of discriminant $\Delta$
 is solved by exploiting the Gauss reduction algorithm, and this solution has deterministic  polynomial complexity $O((\ln p)^\alpha)$,   
with $\alpha \leq 9$, when  Schoof's algorithm is used to compute the square root of $\Delta$ modulo $p$. %with complexity $O(\Delta^{1/2}(\log p)^6)$.
If the problem is to represent $p$ by a specific quadratic form $Q(x,y)$ of discriminant $\Delta$,
once the representation of $p$ by $Q_G(x,y)$ has been obtained, it is necessary to find (provided that it exists) the linear transformation between $Q(x,y)$ and $Q_G(x,y)$. 
At this point the cases of negative and positive $\Delta$s are slightly different. 
When $\Delta$ is negative, the reduced form of $Q(x,y)$ coincides with $Q_G(x,y)$ and the related linear transformation solves the problem. When $\Delta$ is positive, the set of reduced forms may contain more than a single quadratic form, thus a  further step is required to find the linear transformation sending the  reduced form of $Q(x,y)$ to $Q_G(x,y)$. \\
An alternative way fto compute the representations of prime numbers has also been shown, which is equivalent to finding an integer root of a polynomial. This solution is particularly attractive when the class number is small and in any case avoids searching within a set of reduced forms. 

\noindent
As regard the problem of deciding which equivalent class of quadratic forms represents a given prime, a computational algorithm, based on Jacobi characters and the splitting of Hilbert class polynomials, is described.

\noindent
 Lastly, the most tricky, but most useful, computational problem remains open: that of deciding which quadratic forms of discriminant $\Delta$ represent a given composite $m$ without knowing its factorization, and of finding the corresponding representations, i.e. solving equation (\ref{main1}) when solvable.
 The observation that solving this equation may imply easy factorization of composite $m$ \cite{huber} places the problem in a perspective relevant to several of today's applications. 
The factoring viability is generally valid only limited to positive definite quadratic forms that are principal, and in these cases the proof, based on the
 observation that a composite $m$ has a multiplicity of representations, is almost straightforward.

\paragraph{Acknowledgement.}The authors wish to thank Professor Massimo Giulietti (University of Perugia), and Professor Carmelo Interlando (San Diego State University) for many useful comments and suggestions.

{ \begin{singlespace}

\end{singlespace}
}

\pagebreak

%\appendix
\begin{center}
{\bf  Appendix A: Gauss's composition}
\end{center}

%\paragraph{Appendix A - }.
\noindent
Given two quadratic forms   $(a_1,b_1, c_1)$ and $(a_2,b_2, c_2)$ with the same discriminant
 $\Delta$, the composed form  is
$$  a_3 = \frac{a_1a_2}{d^2} \hspace{5mm}, \hspace{5mm} b_3=B\hspace{5mm}, \hspace{5mm}  
  c_3= \frac{b_3^2-\Delta}{4a_3}     \,\,,    \mbox{where}  $$
\begin{enumerate}
   \item $\beta=\frac{b_1+b_2}{2}$, $n=\gcd\{a_1,a_2,\beta \}$, and $t,u,v$ are chosen to satisfy
$   d=ta_1+ua_2+v\beta   $.
   \item $B=\frac{a_1b_2t+ua_2b_1+v(b_1b_2+\Delta)/2}{n}$  
\end{enumerate}
$x_3$ and $y_3$ are given by the equation
$$  \left( \begin{array}{c}
                       x_3  \\
                       y_3  \\
                \end{array}  \right) =
 \left( \begin{array}{cccc}
                       n  & \displaystyle    \frac{(b_2-B)n}{2a_2} &  \displaystyle   \frac{(b_2-B)n}{2a_1}   &  \displaystyle   \frac{[b_1b_2+\Delta-B(b_1+b_2)]n}{4a_1a_2}   \\
     &  &  &  \\
                       0  & \displaystyle    \frac{a_1}{n} & \displaystyle    \frac{a_2}{n}   & \displaystyle    \frac{b_1+b_2}{2n}  \\
                \end{array}  \right)
 \left( \begin{array}{c}
                       x_1x_2  \\
                       x_1y_2  \\
                       y_1x_2  \\
                       y_1 y_2
                \end{array}  \right)
$$

\noindent
By the composition rule, the set of quadratic form classes is a group, that is, composing any quadratic form of one class with any quadratic form of another class, a quadratic form of a third, and the same, class is always obtained. This group is isomorphic to the ideal class group of the quadratic field  $\mathbb Q(\sqrt{\Delta})$.

\vspace{6mm}

\begin{center}
{\bf   Legend, Tables \ref{roma2} - \ref{tab6} }
\end{center}

\vspace{2mm}

\begin{tabular}{lcl} \hline
$\Delta$ &~~& Field discriminant \\
$\omega$ &~~& Integral basis element \\
$Q_i$ &~~& Reduced forms \\
$h_{\mathbb{K}}(x)$ &~~&  Hilbert class polynomial \\
$\pi_a$ &~~& $\langle \pi_a \rangle$ is equal to the non-principal ideal rised to some factor of the the class number \\
%$\jmath_0$  &~~& $\jmath$-invariant,~  root of the Hilbert class polynomial \\
%$\mathbb{E}(\mathbb{L} \mid j_0)$ &~~& Elliptic curve of given $\jmath$-invariant  \\ \hline
  System  &~~&  Diophantine system from Theorem \ref{lab2}  \\ \hline
\end{tabular}

\begin{table}
{\scriptsize
\begin{center}
\begin{tabular}{|c|c|c|c|c|c|}    \hline
$D$  &$\Delta$& $\omega$                       & $Q_i$                       & Ideals                                         & $h_{\mathbb{K}}(x)$                        \\  \hline
-5     & $-20$   & $\sqrt{-5}$                      & $x^2+5y^2$             & $\langle 1, \omega \rangle$        & $x^2-1264000x-681472000$                                                    \\  
        &              &                                        &$2x^2+2xy+3y^2$   & $\langle 2,1+\omega \rangle$     &                                                             \\ \hline
-6     & $-24$   & $\sqrt{-6}$                      & $x^2+6y^2$             & $\langle 1, \omega \rangle$        & $x^2-4834944x+$                                                   \\  
        &              &                                        &$2x^2+3y^2$            & $\langle 2,\omega \rangle$        &  $+14670139392$                                                    \\ \hline
-10   & $-40$   & $\sqrt{-10}$                    & $x^2+10y^2$            & $\langle 1, \omega \rangle$       &$x^2 - 425692800x +$                                              \\
        &              &                                        &$2x^2+5y^2$             & $\langle 2,\omega \rangle$       & $+9103145472000$                                                  \\ \hline             
-13   & $-52$   & $\sqrt{-13}$                    & $x^2+13y^2$            & $\langle 1, \omega \rangle$       & $x^2 - 6896880000x -$                                             \\
        &              &                                        &$2x^2+2xy+7y^2$    &  $\langle 2,1+\omega \rangle$   & $- 567663552000000$                                              \\ \hline
-15   & $-15$   & $\frac{1+\sqrt{-15}}{2}$ & $x^2+xy+4y^2$        & $\langle 1, \omega \rangle$       & $x^2 + 191025x -$                                                     \\ 
        &              &                                        &$2x^2+xy+2y^2$      & $\langle 2,\omega \rangle$        & $- 121287375$                                                           \\ \hline
-22   & $-88$   & $\sqrt{-22}$                    & $x^2+22y^2$            & $\langle 1, \omega \rangle$       & $x^2 - 6294842640000x +$                                       \\ 
        &              &                                        &$2x^2+11y^2$          &  $\langle 2,\omega \rangle$        & $+ 15798135578688000000$                                    \\ \hline  
-35   & $-35$   & $\frac{1+\sqrt{-35}}{2}$ & $x^2+xy+9y^2$        & $\langle 1, \omega \rangle$        &  $x^2 + 117964800x -$                                             \\
        &              &                                        &$3x^2+xy+3y^2$      & $\langle 3,\omega \rangle$        &   $- 134217728000$                                                    \\ \hline
-37   & $-148$ & $\sqrt{-37}$                    & $x^2+37y^2 $           & $\langle 1, \omega \rangle$       &  $x^2 - 39660183801072000x -$                                \\ 
        &              &                                       &$2x^2+2xy+19y^2$   & $\langle 2,1+\omega \rangle$    &  $ - 7898242515936467904000000$                           \\ \hline  
-51   & $-51$   &$\frac{1+\sqrt{-51}}{2}$  & $x^2+xy+13y^2$      & $\langle 1, \omega \rangle$        & $x^2 + 5541101568x + $                                            \\ 
        &              &                                        &$3x^2+3xy+5y^2$    & $\langle 3,1+\omega \rangle$    &  $+ 6262062317568$                                                   \\ \hline
-58   & $-232$ & $\sqrt{-58}$                    & $x^2+58y^2$            & $\langle 1, \omega \rangle$       &  $x^2 - 604729957849891344000x +$                        \\  
        &              &                                        &$2x^2+29y^2$          & $\langle 2,\omega \rangle$         &  $+ 14871070713157137145512000000000$            \\ \hline
-91   & $-91$   & $\frac{1+\sqrt{-91}}{2}$  & $x^2+xy+23y^2$     & $\langle 1, \omega \rangle$        &  $x^2 + 10359073013760x -$                                     \\  
        &              &                                        &$5x^2+3xy+5y^2$    & $\langle 5,1+\omega \rangle$     &  $- 3845689020776448$                                             \\ \hline  
-115 & $-115$ & $\frac{1+\sqrt{-115}}{2}$& $x^2+xy+29y^2$     & $\langle 1, \omega \rangle$        &  $x^2 + 427864611225600x +$                                  \\  
        &              &                                        &$5x^2+5xy+7y^2$    & $\langle 5,2+\omega \rangle$     &  $+ 130231327260672000$                                        \\ \hline
-123 & $-123$ & $\frac{1+\sqrt{-123}}{2}$&$x^2+xy+31y^2$      & $\langle 1, \omega \rangle$       &   $x^2 + 1354146840576000x +$                                 \\  
        &              &                                        &$3x^2+3xy+11y^2$  & $\langle 3,1+\omega \rangle$    &   $+ 148809594175488000000$                                   \\ \hline
-187 & $-187$ & $\frac{1+\sqrt{-187}}{2}$&$x^2+xy+47y^2$      & $\langle 1, \omega \rangle$       &    $x^2 + 4545336381788160000x -$                            \\  
        &              &                                        &$7x^2+3xy+7y^2$    & $\langle 7,1+\omega \rangle$    &    $- 3845689020776448000000$                                  \\ \hline
-235 & $-235$ & $\frac{1+\sqrt{-235}}{2}$&$x^2+xy+59y^2$      & $\langle 1, \omega \rangle$       &    $x^2 + 823177419449425920000x +$                       \\  
        &              &                                        &$5x^2+5xy+13y^2$  & $\langle 5,2+\omega \rangle$    &   $+ 11946621170462723407872000$                          \\ \hline
-267 & $-267$ & $\frac{1+\sqrt{-267}}{2}$&$x^2+xy+67y^2$      & $\langle 1, \omega \rangle$       &  $x^2 + 19683091854079488000000x +$                     \\  
        &              &                                        &$3x^2+3xy+23y^2$  & $\langle 3,1+\omega \rangle$    &  $+ 531429662672621376897024000000$                   \\ \hline
-403 & $-403$ & $\frac{1+\sqrt{-403}}{2}$&$x^2+xy+101y^2$    & $\langle 1, \omega \rangle$       &  $x^2 + 2452811389229331391979520000x -$             \\  
        &              &                                        &$11x^2+9xy+11y^2$& $\langle 11,4+\omega \rangle$  &  $- 108844203402491055833088000000$                     \\ \hline
-427 & $-427$ & $\frac{1+\sqrt{-427}}{2}$&$x^2+xy+107y^2$    & $\langle 1, \omega \rangle$        & $x^2 + 15611455512523783919812608000x +$           \\  
        &              &                                        &$7x^2+7xy+17y^2$  &  $\langle 7,3+\omega \rangle$    & $+ 155041756222618916546936832000000$               \\ \hline
\end{tabular}
\end{center}
}
\caption{Imaginary quadratic fields $\mathbb K= \mathbb Q(\sqrt{D})$ of class number 2}

  \label{roma2}
\end{table}

\begin{table}
{\scriptsize
\begin{center}
\begin{tabular}{|c|c|c|c|c|}    \hline
$D$  &$\Delta$& $\omega$                       &$\pi_a$& Systems                      \\  \hline
-5     & $-20$   & $\sqrt{-5}$                      &2 & $2x^2+2xy-2y^2= u$            \\  
        &              &                                        & & $2xy+y^2=v$                           \\ \hline
-6     & $-24$   & $\sqrt{-6}$                      & 2&  $2x^2-3y^2= u$                                                 \\  
        &              &                                        & &  $2xy=v$                                                  \\ \hline
-10   & $-40$   & $\sqrt{-10}$                    &2 &  $2x^2-5y^2= u$                                             \\
        &              &                                        & &  $2xy=v$                                                \\ \hline             
-13   & $-52$   & $\sqrt{-13}$                    & 2 &$2x^2+2xy-6y^2= u$    \\
        &              &                                        &  &$2xy+y^2=v$                                            \\ \hline
-15   & $-15$   & $\frac{1+\sqrt{-15}}{2}$ & $\omega$  &$x^2+4xy= u$                                  \\ 
        &              &                                        &   &$-x^2+y^2=v$                                                        \\ \hline
-22   & $-88$   & $\sqrt{-22}$                    &  2 &$2x^2-11y^2= u$                                    \\ 
        &              &                                        &   &$2xy=v$                                 \\ \hline  
-35   & $-35$   & $\frac{1+\sqrt{-35}}{2}$ & $\omega$   &$x^2+6xy= u$                                          \\
        &              &                                        &  &$-x^2+y^2=v$                                                  \\ \hline
-37   & $-148$ & $\sqrt{-37}$                    & 2 &$2x^2+2xy-18y^2= u$                              \\ 
        &              &                                       &   &$2xy+y^2=v$                        \\ \hline  
-51   & $-51$   &$\frac{1+\sqrt{-51}}{2}$  & 3  &$3x^2+2xy-4y^2= u$                                          \\ 
        &              &                                        &  &$2xy+y^2=v$                                                 \\ \hline
-58   & $-232$ & $\sqrt{-58}$                    &  2 &$2x^2-29y^2= u$                     \\  
        &              &                                        &  &$2xy=v$          \\ \hline
-91   & $-91$   & $\frac{1+\sqrt{-91}}{2}$  & $1+\omega$ &$2x^2+10xy+y^2= u$                                   \\  
        &              &                                        &  &$-x^2+y^2=v$                     \\ \hline  
-115 & $-115$ & $\frac{1+\sqrt{-115}}{2}$& 5&$5x^2+4xy-5y^2= u$                                  \\  
        &              &                                        &  &$2xy+y^2=v$                                       \\ \hline
-123 & $-123$ & $\frac{1+\sqrt{-123}}{2}$& 3 &$3x^2+2xy-10y^2= u$                               \\  
        &              &                                        &  &$2xy+y^2=v$                                 \\ \hline
-187 & $-187$ & $\frac{1+\sqrt{-187}}{2}$& $1+\omega$  &$2x^2+14xy+y^2= u $                          \\  
        &              &                                        &  &$-x^2+y^2=v$       \\ \hline
-235 & $-235$ & $\frac{1+\sqrt{-235}}{2}$& 5  &$5x^2+4xy-11y^2= u$                     \\  
        &              &                                        &  &$2xy+y^2=v$                       \\ \hline
-267 & $-267$ & $\frac{1+\sqrt{-267}}{2}$&3 &$3x^2+2xy-22y^2= u$                    \\  
        &              &                                        & &$2xy+y^2=v$                  \\ \hline
-403 & $-403$ & $\frac{1+\sqrt{-403}}{2}$& $4+\omega$ &$5x^2+22xy+4y^2= u$\\  
        &              &                                        & &$-x^2+y^2=v$                   \\ \hline
-427 & $-427$ & $\frac{1+\sqrt{-427}}{2}$& 7 &$7x^2+6xy-14y^2= u$  \\  
        &              &                                        &  &$2xy+y^2=v$             \\ \hline
\end{tabular}
\end{center}
}
\caption{Imaginary quadratic fields $\mathbb K= \mathbb Q(\sqrt{D})$ of class number 2}

\label{tab3a}
\end{table}

\begin{table}
\begin{center}
\begin{tabular}{|c|c|c|c|c|c|}    \hline
$D$  &$\Delta$& $\omega$                       & $Q_i$                       & Ideals                                           & $h_{\mathbb{K}}(x)$                        \\  \hline
-23   & $-23$   & $\frac{1+\sqrt{-23}}{2}$ & $x^2+xy+6y^2$      & $\langle 1, \omega \rangle$    & \tiny{$x^3 + 3491750x^2 -$}                                                    \\  
         &               &                                           &$2x^2+xy+3y^2$    & $\langle 2,\omega \rangle$     & \tiny{$- 5151296875x + 12771880859375$}                                                       \\
         &               &                                           &                                   &                                                      &                                                                                                                   \\ \hline
-31     & $-31$   & $\frac{1+\sqrt{-31}}{2}$                      & $x^2+xy+8y^2$             & $\langle 1, \omega \rangle$        & \tiny{$x^3 + 39491307x^2 -$}                                                   \\  
        &              &                                        &$2x^2+xy+4y^2$            & $\langle 2,1+ \omega \rangle$        &  \tiny{$- 58682638134x + 1566028350940383$}                                                    \\ 
        &               &                                           &                                   &                                                      &                                                                                                                   \\ \hline
-59   & $-59$   & $\frac{1+\sqrt{-59}}{2}$                    & $x^2+xy+15y^2$            & $\langle 1, \omega \rangle$       &\tiny{$x^3 + 30197678080x^2 -$}                                              \\
        &              &                                        &$3x^2+xy+5y^2$             & $\langle 3,\omega \rangle$       & \tiny{$- 140811576541184x + 374643194001883136$}                                                  \\ 
        &               &                                           &                                   &                                                      &                                                                                                                   \\ \hline            
-83   & $-83$   & $\frac{1+\sqrt{-83}}{2}$                    & $x^2+xy+21y^2$            & $\langle 1, \omega \rangle$       & \tiny{$x^3 + 2691907584000x^2 -$}                                             \\
        &              &                                        &$3x^2+xy+7y^2$    &  $\langle 3,\omega \rangle$   & \tiny{$- 41490055168000000x +$}                                              \\ 
                 &               &                                           &                                   &                                                      &\tiny{$+ 549755813888000000000$}                                                       \\ \hline
-107   & $-107$   & $\frac{1+\sqrt{-107}}{2}$ & $x^2+xy+27y^2$        & $\langle 1, \omega \rangle$       & \tiny{$x^3 + 129783279616000x^2 -$}                                                    \\ 
        &              &                                        &$3x^2+xy+9y^2$      & $\langle 3,\omega \rangle$        & \tiny{$- 6764523159552000000x +$}                                                           \\ 
                 &               &                                           &                                   &                                                      & \tiny{$+ 337618789203968000000000$}                                         \\ \hline
-139   & $-139$   & $\frac{1+\sqrt{-139}}{2}$                    & $x^2+xy+35y^2$            & $\langle 1, \omega \rangle$       & \tiny{$x^3 + 12183160834031616x^2-$}                                        \\ 
        &              &                                        &$5x^2+xy+7y^2$          &  $\langle 5,\omega \rangle$        & \tiny{$- 53041786755137667072x +$}                                    \\ 
                 &               &                                           &                                   &                                                      & \tiny{$+    67408489017571610198016$}                                                 \\ \hline
-211   & $-211$   & $\frac{1+\sqrt{-211}}{2}$ & $x^2+xy+53y^2$        & $\langle 1, \omega \rangle$        &  \tiny{$x^3 + 65873587288630099968x^2 +$}                                             \\
        &              &                                        &$5x^2+3xy+11y^2$      & $\langle 5,1+\omega \rangle$        &   \tiny{$+ 277390576406111100862464x +$}                                                    \\ 
                 &               &                                           &                                   &                                                      &   \tiny{$+5310823021408898698117644288$}                                                \\ \hline
-283   & $-283$ & $\frac{1+\sqrt{-283}}{2}$                    & $x^2+xy+71y^2 $           & $\langle 1, \omega \rangle$       &  \tiny{$x^3 + 89611323386832801792000x^2 +$}                                \\ 
        &              &                                       &$7x^2+5xy+11y^2$   & $\langle 7,2+\omega \rangle$    & \tiny{$+ 90839236535446929408000000x +$}                          \\ 
                 &               &                                           &                                   &                                                      &  \tiny{$+ 201371843156955365376000000000$}                                                     \\ \hline
\end{tabular}
\end{center}
\caption{Imaginary quadratic fields $\mathbb K= \mathbb Q(\sqrt{D})$ of class number 3}

 \label{tab4a}
\end{table}

%
%\end{document}

\begin{table}
\begin{center}
\begin{tabular}{|c|c|c|c|c|c|}    \hline 
$D$  &$\Delta$& $\omega$                       & $Q_i$                       & Ideals                                           & $h_{\mathbb{K}}(x)$                        \\  \hline                  
-307   & $-307$   &$\frac{1+\sqrt{-307}}{2}$  & $x^2+xy+77y^2$      & $\langle 1, \omega \rangle$        & \tiny{$x^3 + 805016812009981390848000x^2 -$}                                           \\ 
        &              &                                        &$7x^2+xy+11y^2$    & $\langle 7,\omega \rangle$    & \tiny{$- 5083646425734146162688000000x +$}                                                   \\
         &               &                                           &                                   &                                                      &  \tiny{$+ 8987619631060626702336000000000$}                                       \\ \hline        
-331   & $-331$ & $\frac{1+\sqrt{-331}}{2}$                    & $x^2+xy+83y^2$            & $\langle 1, \omega \rangle$       &  \tiny{$x^3 + 6647404730173793386463232x^2 +$}                        \\  
        &              &                                        &$5x^2+3xy+17y^2$          & $\langle 5,1+\omega \rangle$         &  \tiny{$+ 368729929041040103875232661504x +$}             \\
         &               &                                           &                                   &                                                      &             \tiny{$+ 56176242840389398230218488594563072$}           \\ \hline        
-379   & $-379$   & $\frac{1+\sqrt{-379}}{2}$  & $x^2+xy+95y^2$     & $\langle 1, \omega \rangle$        &  \tiny{$x^3 + 364395404104624239018246144x^2 -$}                                    \\  
        &              &                                        &$5x^2+xy+19y^2$    & $\langle 5,\omega \rangle$     &  \tiny{$- 121567791009880876719538528321536x +$}                                              \\
                 &               &                                           &                                   &                                                      &  \tiny{$+15443600047689011948024601807415148544$}             \\ \hline  
-499 & $-499$ & $\frac{1+\sqrt{-499}}{2}$& $x^2+xy+125y^2$     & $\langle 1, \omega \rangle$        &  \tiny{$x^3 + 3005101108071026200706725969920x^2 -$}                                  \\  
        &              &                                        &$5x^2+xy+25y^2$    & $\langle 5,\omega \rangle$     &  \tiny{$- 6063717825494266394722392560011051008x +$}                                        \\ 
                 &               &                                           &                                   &                                                      & \tiny{$+4671133182399954782798673154437441310949376$}             \\ \hline
-547 & $-547$ & $\frac{1+\sqrt{-547}}{2}$&$x^2+xy+137y^2$      & $\langle 1, \omega \rangle$       &   \tiny{$x^3 + 81297395539631654721637478400000x^2 -$}                                \\  
        &              &                                        &$11x^2+5xy+13y^2$  & $\langle 11,2+\omega \rangle$    &   \tiny{$ 139712328431787827943469744128000000x +$}                                  \\
                 &               &                                           &                                   &                                                      &  \tiny{$ 83303937570678403968635240448000000000$}          \\ \hline
-643 & $-643$ & $\frac{1+\sqrt{-643}}{2}$&$x^2+xy+161y^2$      & $\langle 1, \omega \rangle$       &    \tiny{$x^3 + 39545575162726134099492467011584000x^2 - $}                            \\  
        &              &                                        &$7x^2+xy+23y^2$    & $\langle 7,\omega \rangle$    &    \tiny{$6300378505047247876499651797450752000000x +$}                                   \\ 
                 &               &                                           &                                   &                                                      &   \tiny{$+ 308052554652302847380880841299197952000000000$}               \\ \hline
-883 & $-883$ & $\frac{1+\sqrt{-883}}{2}$&$x^2+xy+221y^2$      & $\langle 1, \omega \rangle$       &    \tiny{$x^3 + 34903934341011819039224295011933392896000x^2 -$}                      \\  
        &              &                                        &$13x^2+xy+17y^2$  & $\langle 13,\omega \rangle$    &  \tiny{$151960111125245282033875619529124478976000000x +$}                           \\ 
                 &               &                                           &                                   &                                                      &  \tiny{$+167990285381627318187575520800123387904000000000$}                   \\ \hline
-907 & $-907$ & $\frac{1+\sqrt{-907}}{2}$&$x^2+xy+227y^2$      & $\langle 1, \omega \rangle$       &  \tiny{$x^3 + 123072080721198402394477590506838687744000x^2 +$}                    \\  
         &               &                                             &$13x^2+9xy+19y^2$  & $\langle 13,4+\omega \rangle$    &  \tiny{$+ 39181594208014819617565811575376314368000000x+$} \\
         &               &                                             &                                      &                                                      &\tiny{$+ 149161274746524841328545894969274007552000000000$}     \\ \hline

\end{tabular}
\end{center}
\caption{Imaginary quadratic fields $\mathbb K= \mathbb Q(\sqrt{D})$ of class number 3}
  \label{tab5}
\end{table}

%\end{document}

\begin{table}
{\scriptsize
\begin{center}
\begin{tabular}{|c|c|c|c|c|}    \hline
$D$  &$\Delta$ & $\omega$                            &$\pi_a$                 & Systems                      \\  \hline
$-23$     & $-23$   & $\frac{1+\sqrt{-23}}{2}$     &  $2-\omega$        & $x^3-9x^2y-9xy^2+3y^3=u$                                                    \\  
           &               &                                         &                             & $x^3+3x^2y-3xy^2-2y^3=v$                                                       \\ \hline
$-31$     & $-31$   & $\frac{1+\sqrt{-31}}{2}$     & $1-\omega$         & $-12x^2y-18xy^2+y^3=u$                                         \\  
        &              &                                             &                             &  $x^3+3x^2y-3xy^2-4y^3=v$                  \\ \hline
$-59$   & $-59$   & $\frac{1+\sqrt{-59}}{2}$       & $-3-\omega$       &   $-4x^3-15x^2y+15xy^2+10y^3=u$                           \\
        &              &                                             &                             &  $x^3-3x^2y-6xy^2+y^3=v$            \\ \hline             
$-83$   & $-83$   & $\frac{1+\sqrt{-83}}{2}$       & $3-\omega$         &$2x^3-21x^2y-21xy^2+14y^3=u$                         \\
        &              &                                             &                             & $x^3+3x^2y-6xy^2-3y^3=v$     \\ \hline
$-107$   & $-107$   & $\frac{1+\sqrt{-107}}{2}$ &  $-\omega$          & $-x^3-27x^2y+27y^3=u$               \\ 
        &              &                                             &                             &        $x^3-9xy^2-y^3=v$       \\ \hline
$-139$   & $-139$   & $\frac{1+\sqrt{-139}}{2}$ & $10-\omega$       &  $9x^3-21x^2y-42xy^2+7y^3=u$              \\ 
        &              &                                             &                             &      $x^3+6x^2y-3xy^2-3y^3=v$              \\ \hline  
$-211$   & $-211$   & $\frac{1+\sqrt{-211}}{2}$ & $9-\omega$         &$8x^3-27x^2y-69xy^2+6y^3=u$                               \\
        &              &                                             &                             &  $x^3+6x^2y-3xy^2-5y^3=v$                              \\ \hline
$-283$   & $-283$ & $\frac{1+\sqrt{-283}}{2}$   & $-16-\omega$     & $-17x^3-45x^2y+48xy^2+35y^3=u$                              \\ 
        &              &                                             &                             &  $x^3-6x^2y-9xy^2+y^3=v$   \\ \hline  
$-307$   & $-307$   &$\frac{1+\sqrt{-307}}{2}$  & $7-2\omega$       & $5x^3-66x^2y-33xy^2+33y^3=u$                       \\ 
        &              &                                             &                             &  $2x^3+3x^2y-9xy^2-2y^3=v$                       \\ \hline
$-331$   & $-331$ & $\frac{1+\sqrt{-331}}{2}$   & $-6-\omega$       & $-7x^3-54x^2y+39xy^2+69y^3=u$                 \\  
        &              &                                             &                             & $x^3-3x^2y-12xy^2+y^3=v$\\ \hline
$-379$   & $-379$   & $\frac{1+\sqrt{-379}}{2}$ & $-5-\omega$       &  $-6x^3-57x^2y+57xy^2+76y^3=u$                           \\  
        &              &                                             &                            &  $x^3-3x^2y-12xy^2+3y^3=v$                           \\ \hline  
$-499$ & $-499$ & $\frac{1+\sqrt{-499}}{2}$     & $-\omega$          & $-x^3-75x^2y+125y^3=u$               \\  
        &              &                                             &                            & $x^3-15xy^2-y^3=v$                   \\ \hline
$-547$ & $-547$ & $\frac{1+\sqrt{-547}}{2}$     & $29-2\omega$    & $-27x^3-60x^2y-123xy^2+5y^3=u$         \\  
        &              &                                             &                            &  $2x^3+9x^2y-3xy^2-4y^3=v$          \\ \hline
$-643$ & $-643$ & $\frac{1+\sqrt{-643}}{2}$     & $14-\omega$      & $13x^3-69x^2y-138xy^2+69y^3=u$     \\  
        &              &                                             &                            &  $x^3+6x^2y-9xy^2-7y^3=v$         \\ \hline
$-883$ & $-883$ & $\frac{1+\sqrt{-883}}{2}$     & $13+3\omega$   & $16x^3+153x^2y-51xy^2-68y^3=u$       \\  
        &              &                                             &                            &  $-3x^3+3x^2y+12xy^2-y^3=v$         \\ \hline
$-907$ & $-907$ & $\frac{1+\sqrt{-907}}{2}$     & $-14+3\omega$  & $-11x^3+147x^2y+150xy^2-37y^3=u$      \\  
        &              &                                             &                            &  $-3x^3-6x^2y+9xy^2+5y^3=v$  \\ \hline
\end{tabular}
\end{center}
}
\caption{Imaginary quadratic fields $\mathbb K= \mathbb Q(\sqrt{D})$ of class number 3}
\label{tab6}
\end{table}

\end{document}